\documentclass[10pt,reqno]{amsart}

\usepackage{color}

\newtheorem{thm}{Theorem}[section]
\newtheorem{defn}[thm]{Definition}
\newtheorem{corollary}[thm]{Corollary}
\newtheorem{lemma}[thm]{Lemma}

\newtheorem{remark}[thm]{Remark}

\newcommand{\pf}{\noindent{\bf Proof.} }
\def\qed{{\hfill $\Box$ \bigskip}}

\def\Xint#1{\mathchoice
{\XXint\displaystyle\textstyle{#1}}%
{\XXint\textstyle\scriptstyle{#1}}%
{\XXint\scriptstyle\scriptscriptstyle{#1}}%
{\XXint\scriptscriptstyle\scriptscriptstyle{#1}}%
\!\int}
\def\XXint#1#2#3{{\setbox0=\hbox{$#1{#2#3}{\int}$}
\vcenter{\hbox{$#2#3$}}\kern-.5\wd0}}

\def\dashint{\Xint-}

\newcommand\aint{-\hspace{-0.38cm}\int}

\newcommand\cbrk{\text{$]$\kern-.15em$]$}}
\newcommand\opar{\text{\,\raise.2ex\hbox{${\scriptstyle
|}$}\kern-.34em$($}}
\newcommand\cpar{\text{$)$\kern-.34em\raise.2ex\hbox{${\scriptstyle |}$}}\,}

\def\wh{\widehat}

\def\<{\langle}
\def\>{\rangle}
\def\eps{\varepsilon}
\def\E{{\mathbb E}}

\newcommand\bL{\mathbb{L}}
\newcommand\bR{\mathbb{R}}

\newcommand\bH{\mathbb{H}}

\newcommand\bD{\mathbb{D}}
\newcommand\bS{\mathbb{S}}
\newcommand\bM{\mathbb{M}}
\newcommand\bE{\mathbb{E}}
\newcommand\bN{\mathbb{N}}
\newcommand\bP{\mathbb{P}}

\newcommand\cI{\mathcal{I}}
\newcommand\cA{\mathcal{A}}
\newcommand\cB{\mathcal{B}}

\newcommand\cF{\mathcal{F}}
\newcommand\cG{\mathcal{G}}
\newcommand\cH{\mathcal{H}}
\newcommand\cL{\mathcal{L}}
\newcommand\cP{\mathcal{P}}

\newcommand\cM{\mathcal{M}}

\newcommand\cQ{\mathcal{Q}}

\def\R {{\mathbb R}}

\def\wh{\widehat}

\newcommand{\mysection}[1]{\section{#1}
\setcounter{equation}{0}}

\begin{document}

\title[Parabolic littlewood-paley inequaltiy]
{Parabolic  Littlewood-Paley   inequality for   $\phi(-\Delta)$-type operators and applications to Stochastic integro-differential equations}

\author{Ildoo Kim}
\address{Department of Mathematics, Korea University, 1 Anam-dong, Sungbuk-gu, Seoul,
136-701, Republic of Korea} \email{waldoo@korea.ac.kr}

\author{Kyeong-Hun Kim}
\address{Department of Mathematics, Korea University, 1 Anam-dong,
Sungbuk-gu, Seoul, 136-701, Republic of Korea}
\email{kyeonghun@korea.ac.kr}
\thanks{The research of the second
author was supported by Basic Science Research Program through the
National Research Foundation of Korea(NRF) funded by the Ministry of
Education, Science and Technology (20120005158)}

\author{Panki Kim}
\address{Department of Mathematical Sciences and Research Institute of Mathematics,
Seoul National University,
Building 27, 1 Gwanak-ro, Gwanak-gu
Seoul 151-747, Republic of Korea.}
\email{pkim@snu.ac.kr}
\thanks{The research of the third
author was supported by Basic
Science Research Program through the National Research Foundation of
Korea(NRF) grant funded by the Korea
government(MEST)(2011-000251)}



\begin{abstract}
In this paper we prove a  parabolic version of the Littlewood-Paley  inequality  (\ref{littlewood-paley bernstein})
for the operators of the type $\phi(-\Delta)$, where $\phi$ is a  Bernstein function. As an application, we construct an $L_p$-theory for  the stochastic integro-differential equations of the type $du=(-\phi(-\Delta)u+f)\,dt +g\,dW_t$.

\end{abstract}

\maketitle

\mysection{Introduction}

The operators  we are  considering in this article are certain functions of the Laplacian. To be more precise,  recall that a function $\phi:(0,\infty)\to (0,\infty)$ such that $\phi(0+)=0$ is called a Bernstein function if it is of the form
\begin{equation*}\label{e:bernstein-function}
\phi(\lambda)=b \lambda +\int_{(0,\infty)}(1-e^{-\lambda t})\, \mu(dt)\, ,\quad \lambda >0\, ,
\end{equation*}
where $b\ge 0$ and $\mu$ is a measure on $(0,\infty)$ satisfying $\int_{(0,\infty)}(1\wedge t)\, \mu(dt)<\infty$, called the L\'evy measure.  By Bochner's functional calculus, one can define the operator
$\phi(\Delta):=-\phi(-\Delta)$ on
$C_b^2(\R^d)$,  which turns out to be an integro-differential operator
\begin{equation}\label{e:phirep}
b \Delta f(x)+\int_{\R^d}\left(f(x+y)-f(x)-\nabla f(x) \cdot y {\mathbf 1}_{\{|y|\le 1\}}\right)\, J(y)\, dy \, ,
\end{equation}
where $J(x)=j(|x|)$ with $j:(0,\infty)\to (0,\infty)$ given by
$$
j(r)=\int_0^{\infty} (4\pi t)^{-d/2} e^{-r^2/(4t)}\, \mu(dt)\, .
$$

It is also known that the operator  $\phi(\Delta)$ is the infinitesimal generator of  the $d$-dimensional subordinate Brownian motion. Let
 $S=(S_t)_{t\ge 0}$ be a subordinator  (i.e. an
increasing L\'evy process satisfying $S_0=0$) with Laplace exponent $\phi$,
and let $W=(W_t)_{t\ge 0}$ be a Brownian motion in $\R^d$, $d\ge 1$,  independent of $S$ with
$\E_x\left[e^{i\xi(W_t-W_0)}\right]=e^{-t{|\xi|^2}}, \xi\in \R^d, t>0$. Then
     $X_t:=W_{S_t}$, called the subordinate Brownian motion,   is a rotationally invariant L\'evy process in $\R^d$
with characteristic exponent $\phi(|\xi|^2)$, and for
 any $f\in C^2_b(\bR^d)$
\begin{equation}\label{e:ig}
\phi(\Delta)f(x)=\lim_{t \to 0} \frac{1}{t}[\bE_x f(X_t)-f(x)].
\end{equation}
For instance, by taking $\phi(\lambda)= \lambda^{\alpha/2}$ with $\alpha \in (0, 2)$, we get the fractional laplacian $\Delta^{\alpha/2}:=-(-\Delta)^{\alpha/2}$  which is
 the infinitesimal generator of a rotationally
symmetric $\alpha$-stable process in $\R^d$.

 In this article we  prove a  parabolic Littlewood-Paley inequality for $\phi(\Delta)$:
\begin{thm}
                                           \label{main theorem}
Let $\phi$ be a Bernstein function, $T_t$ be the semigroup corresponding to $\phi(\Delta)$ and  $H$ be  a Hilbert space. Suppose that $\phi$ satisfies

\noindent
{\bf (H1)}$:$ $\exists$ constants $0<\delta_1\le \delta_2 <1$ and $a_1, a_2>0$  such that
\begin{equation*}\label{e:H1}
a_1\lambda^{\delta_1} \phi(t) \le \phi(\lambda t) \le a_2 \lambda^{\delta_2} \phi(t), \quad \lambda \ge 1, t \ge 1\, ;
\end{equation*}
\noindent
{\bf (H2)}$:$ $\exists$  constants $0<\delta_3 \le 1$ and $a_3>0$  such that
\begin{equation*}\label{e:H2}
 \phi(\lambda t) \le a_3 \lambda^{\delta_3} \phi(t), \quad \lambda \le 1, t \le 1\, .
\end{equation*}
Then for any $p\in [2,\infty), \, T\in (0,\infty)$ and  $f\in   C_0^\infty (\bR^{d+1},H)$,
\begin{eqnarray}
                     \label{littlewood-paley bernstein}
\int_{\bR^d} \int_0^T [\int_{0}^t |\phi(\Delta)^{1/2}T_{t-s}f(s,\cdot)(x)|_H^2 ds]^{p/2}dt dx \leq N
\int_{\bR^d} \int_0^T |f(t,x)|^p_H~dtdx,
\end{eqnarray}
where the constant $N$ depends only on
$d,p,T, a_i$ and $\delta_i$ $(i=1,2,3)$.
\end{thm}

{\bf (H1)} is a condition on the asymptotic behavior of $\phi$ at infinity and it governs the behavior of the corresponding subordinate Brownian motion $X$ for small time and small space. {\bf (H2)} is a condition about the asymptotic behavior of $\phi$ at zero and it governs the behavior of the corresponding subordinate Brownian motion $X$ for large time and large space.
Note that it follows from the second inequality in {\bf (H1)} that $\phi$ has no drift, i.e., $b=0$ in \eqref{e:phirep}. It also follows from  {\bf (H2)} that $\phi(0+)=0$.

Using the tables at the end of \cite{SSV}, one can construct a lot of explicit
examples of Bernstein functions satisfying {\bf (H1)}--{\bf (H2)}.
Here are a few of them:
\begin{itemize}
\item[(1)] $\phi(\lambda)=\lambda^\alpha + \lambda^\beta$, $0<\alpha<\beta<1$;
\item[(2)] $\phi(\lambda)=(\lambda+\lambda^\alpha)^\beta$, $\alpha, \beta\in (0, 1)$;
\item[(3)] $\phi(\lambda)=\lambda^\alpha(\log(1+\lambda))^\beta$, $\alpha\in (0, 1)$,
$\beta\in (0, 1-\alpha)$;
\item[(4)] $\phi(\lambda)=\lambda^\alpha(\log(1+\lambda))^{-\beta}$, $\alpha\in (0, 1)$,
$\beta\in (0, \alpha)$;
\item[(5)] $\phi(\lambda)=(\log(\cosh(\sqrt{\lambda})))^\alpha$, $\alpha\in (0, 1)$;
\item[(6)] $\phi(\lambda)=(\log(\sinh(\sqrt{\lambda}))-\log\sqrt{\lambda})^\alpha$, $\alpha\in (0, 1)$.
\end{itemize}
For example, the subordinate Brownian motion corresponding to the example (1) $\phi(\lambda)=\lambda^\alpha + \lambda^\beta$ is the sum of two independent symmetric $\alpha$ and $\beta$ stable processes. In this case the characteristic exponent is
$$
\Phi(\theta)=|\theta|^{\alpha}+|\theta|^{\beta}\, , \ \theta \in \R^d\, ,\quad 0<\beta <\alpha <2\, ,
$$
and  its infinitesimal generator is $-(-\Delta)^{\beta/2}-(-\Delta)^{\alpha/2}$.

We remark here that relativistic stable processes satisfy {\bf (H1)}--{\bf (H2)} with $\delta_3=1$;
Suppose that $\alpha\in (0, 2)$, $m >0$ and define
$$
\phi_m(\lambda)=(\lambda+m^{2/\alpha})^{\alpha/2}-m.
$$
The subordinate Brownian motion corresponding to $\phi_m$ is
a relativistic $\alpha$-stable process
on $\bR^d$ with mass $m$ whose characteristic
function is given by
$$
 \exp(-t ( (|\xi|^2+ m^{2/\alpha}
    )^{\alpha/2}-m) ), \qquad \xi \in \bR^d.
$$
 The infinitesimal generator is
$m-(m^{2/\alpha} -\Delta)^{\alpha /2}$.

Note that when $m=1$, this
infinitesimal generator reduces to $1-(1 -\Delta)^{\alpha /2}$. Thus
the $1$-resolvent kernel of  the relativistic $\alpha$-stable
process with mass $1$ on $\bR^d$  is just the Bessel potential kernel. When $\alpha=1$, the
infinitesimal generator reduces to the so-called free relativistic
Hamiltonian $ m - \sqrt{-\Delta + m^{2}}$. The operator $m -
\sqrt{-\Delta + m^{2}}$ is very important in mathematical physics
due to its application to relativistic quantum mechanics.  We emphasize that the present article covers this  case.

\vspace{1mm}

 The parabolic Littlewood-Paley inequality  (\ref{littlewood-paley bernstein}) was first proved by Krylov (\cite{Kr01, kr94}) for the case $\phi(\Delta)=\Delta$ with $N=N(p)$ depending only on $p$.
  In this case,  if $f$ depends only on $x$ and $H=\bR$ then (\ref{littlewood-paley bernstein}) leads to the
  the classical (elliptic) Littlewood-Paley
inequality (cf. \cite{Ste}):
\begin{equation}
             % \label{LP}
\int_{\bR^d} \left(\int^{\infty}_0 |\nabla T_{t}
f|^2dt\right)^{p/2} dx\leq N(p)\|f\|^p_p, \quad \quad \forall \, \, f\in L_p(\bR^d). \nonumber
\end{equation}

\noindent

Recently, (\ref{littlewood-paley bernstein}) was  proved for the fractional Laplacian $\Delta^{\alpha/2}$, $\alpha\in (0,2)$, in \cite{CL, Ildoo}. Also, in \cite{MP} similar result was proved for the case $J=J(t,y)=m(t,y)|y|^{-d-\alpha}$ in (\ref{e:phirep}), where $\alpha\in (0,2)$ and $m(t,y)$ is a bounded smooth function satisfying
$m(t,y)=m(t,y|y|^{-1})$ (i.e. homogeneous of degree zero) and $m(t,y)>c>0$ on a set $\Gamma \subset S^{d-1}$ of a positive Lebesgue measure.
%If $|y|^{d+\alpha} J(y)$ is bounded for some $\alpha \in (0,2)$, then this operator is already treated in \cite{MP} (clearly, fractional Laplacian $\Delta^{\alpha/2}$ is contained in this case). But, under our assumptions it could happen that $|y|^{d+\alpha} J(y)$ is unbounded for any $\alpha \in (0,2)$. For instance, if $\delta_1=\delta_2$ and $\delta_3 > \delta_1$ then from Lemma \ref{l:j-g-upper} one can easily check that $|y|^{d+\alpha} J(y)$ is unbounded for any $\alpha \in (0,2)$ ({\bf{check}}).
We note that even the case $\phi(\lambda)=\lambda^{\alpha}+\lambda^{\beta}$ ($\alpha\neq \beta$) is not covered in \cite{MP}.

\vspace{1mm}

Our motivation of studying (\ref{littlewood-paley bernstein}) is that  (\ref{littlewood-paley bernstein}) is the key estimate for the  $L_p$-theory of the corresponding stochastic partial differential equations. For example, Krylov's result
(\cite{Kr01, kr94}) for $\Delta$ is related to the $L_p$-theory of the second-order stochastic partial differential equations. Below we briefly explain the reason for this. See \cite{KK, Kr99} or Section \ref{section application} of this article for more details.
Consider the stochastic integro-differential equation
\begin{equation}
                 \label{eqn 0}
 du=(\phi(\Delta) u +h) \,dt +\sum_{k=1}^{\infty}f^kdw^k_t,
\quad u(0,x)=0.
\end{equation}
 Here
 $f=(f^1,f^2,\cdots)$ is an
$\ell_2$-valued random function of $(t,x)$, and $w^k_t$ are
independent one-dimensional Wiener processes defined on a probability space $(\Omega,P)$. Considering $u-w$, where $w(t):=\int^t_0 T_{t-s}h(s)ds$, we may assume $h=0$ (see Section \ref{section application}).
It turns out that if
$f=(f^1,f^2,\cdots)$ satisfies certain measurability condition, the
solution of this problem is given by
\begin{equation}
                                  % \label{eqn 222}
u(t,x)= \sum_{k=1}^{\infty}\int^t_0
T_{t-s}f^k(s,\cdot)(x) dw^k_s. \nonumber
\end{equation}
 By Burkholder-Davis-Gundy inequality,
\begin{eqnarray}
&&\bE \int^T_0\|\phi(\Delta)^{1/2}u(t,\cdot)\|^p_{L_p}dt \nonumber \\
&\leq&
N(p)\, \bE
\int^T_0\int_{\bR^d}\left[\int^t_0|\phi(\Delta)^{1/2}T_{t-s}f(s,\cdot)(x)|^2_{\ell_2}  \label{eqn 333}
ds\right]^{p/2}dxdt.
\end{eqnarray}
Actually if $f$ is not random, then $u$ becomes a  Gaussian process and  the reverse inequality of (\ref{eqn 333}) also
holds. Thus to prove $\phi(\Delta)^{1/2}u\in L_p$ and to get a
legitimate start of the $L_p$-theory of equation (\ref{eqn 0}),
one has to estimate the right-hand side of (\ref{eqn 333}) (or the left-hand side of (\ref{littlewood-paley bernstein})). We will also see that (\ref{littlewood-paley bernstein}) yields  the uniqueness and existence  of   equation (\ref{eqn 0}) in certain Banach spaces.

\vspace{1mm}

 The key of our approach  is  estimating the  sharp function $(v)^{\sharp}(t,x)$  of  $v(t,x):=[\int_{0}^t |\phi(\Delta)^{1/2}T_{t-s}f(s,\cdot)(x)|_H^2 ds]^{1/2}$:
\begin{equation}
              \label{eqn sharp}
 (v)^{\sharp}(t,x):=\sup_{(t,x)\in \cQ}\aint_{\cQ}|v-v_{\cQ}|dtdx,
 \end{equation}
 where $v_{\cQ} :=\aint_{\cQ} v\;dxdt$ is the average of $v$ over $Q$ and
the supremum is taken for all cubes $Q$  containing $(t,x)$ of the type $\cQ_c(r,y):=(r-\phi(c^{-2})^{-1},r+\phi(c^{-2})^{-1})\times B_c(y)$.
We control $(v)^{\sharp}(t,x)$ in terms of the  maximal functions of $|f|_{H}$, and then apply
    Fefferman-Stein and Hardy-Littlewood theorems to prove (\ref{littlewood-paley bernstein}).
    The operators considered in \cite{Kr01, Ildoo, MP} have  simple scaling properties, and so to estimate the mean oscillation $\aint_{\cQ}
|v-v_{\cQ}|dtdx$ in (\ref{eqn sharp}), it was enough to consider the only case $\cQ=\cQ_1(0,0)$, that is the case $c=1$ and $(r,y)=(0,0)$.
However, in our case, due to the lack of the scaling property,  it is needed to consider the mean oscillation  $\aint_{\cQ}
|v-v_{\cQ}|dtdx$ on every $Q_c(r,y)$ containing $(t,x)$. This causes serious difficulties as can be seen in the proofs of Lemmas \ref{co1}--\ref{2-5}.
 Our estimation of $\aint_{\cQ}
|v-v_{\cQ}|dtdx$  relies on the upper bounds of $\phi(\Delta)^{n/2}D^{\beta} p(t,x)$, which are  obtained in this article. Here  $\beta$ is an arbitrary multi-index, $n=0,1,2,\cdots$ and $p(t,x)$ is the density of the semigroup $T_t$ corresponding to $\phi(\Delta)$.

\vspace{1mm}

The article is organized as follows. In Section \ref{s:p}
we give upper bounds of the density $p(t,x)$. Section \ref{s:2} contains various properties of
Bernstein functions and subordinate Brownian motions. In Section \ref{s:3} we establish upper bounds of the fractional derivatives of  $p(t,x)$ in terms of $\phi$. Using these estimates we give the proof of of Theorem \ref{main theorem} in Section \ref{s:proof}. In Section  \ref{section application} we apply Theorem \ref{main theorem} and construct an  $L_p$-theory for equation (\ref{eqn 0}).

\vspace{1mm}

We finish the introduction with some notation. As usual $\bR^{d}$ stands for the Euclidean space of points
$x=(x^{1},...,x^{d})$,  $B_r(x) := \{ y\in \bR^d : |x-y| < r\}$  and
$B_r
 :=B_r(0)$.
 For $i=1,...,d$, multi-indices $\beta=(\beta_{1},...,\beta_{d})$,
$\beta_{i}\in\{0,1,2,...\}$, and functions $u(x)$ we set
$$
u_{x^{i}}=\frac{\partial u}{\partial x^{i}}=D_{i}u,\quad
D^{\beta}u=D_{1}^{\beta_{1}}\cdot...\cdot D^{\beta_{d}}_{d}u,
\quad|\beta|=\beta_{1}+...+\beta_{d}.
$$
We write $u\in C^{\infty}_0(X,Y)$ if $u$ is a $Y$-valued infinitely differentiable function defined on $X$ with compact support. By $C^2_b(\bR^d)$ we denote the space of twice continuously differentiable functions on $\bR^d$ with bounded derivatives up to order $2$.
We use  ``$:=$" to denote a definition, which
is  read as ``is defined to be". We denote $a \wedge b := \min \{ a, b\}$, $a \vee b := \max \{ a, b\}$. If we write $N=N(a, \ldots, z )$,
this means that the constant $N$ depends only on $a, \ldots , z$. The constant $N$ may change  from location to location, even within a line. By $\cF$ and $\cF^{-1}$ we denote the Fourier transform and the inverse Fourier transform, respectively. That is, for a  suitable function $f$,
$\cF(f)(\xi) := \int_{\bR^{d}} e^{-i x \cdot \xi} f(x) dx$ and $\cF^{-1}(f)(x) := \frac{1}{(2\pi)^d}\int_{\bR^{d}} e^{ i\xi \cdot x} f(\xi) d\xi$.
Finally, for a Borel
set $A\subset \bR^d$, we use $|A|$ to denote its Lebesgue
measure.

\mysection{Upper bounds of $p(t,x)$}\label{s:p}

In this section we give upper bounds of the density $p(t,x)$
of the semigroup $T_t$ corresponding to $\phi(\Delta)$.
 We give the result under slightly more general setting.
We will assume that $Y$ is a  rotationally
symmetric L\'evy process with L\'evy exponent
$\Psi_Y(\xi)$.  Because of rotational symmetry, the function $\Psi_Y$ is positive and depends on $|\xi|$ only. Accordingly, by a slight abuse of notation we write $\Psi_Y(\xi)=\Psi_Y(|\xi|)$ and get
\begin{equation}\label{e:Psi}
\E_x\left[e^{i\xi\cdot(Y_t-Y_0)}\right]=e^{-t\Psi_Y(|\xi|)},
\quad \quad \mbox{ for every } x\in \R^d \mbox{ and } \xi\in \R^d.
\end{equation}
We assume that the transition probability $\bP(Y_t \in dy)$ is absolutely continuous with respect to Lebesgue measure in $\R^d$.
Thus there is a function $p_Y(t,r)$, $t>0, r \ge 0$ such that
$$
\bP(Y_t \in dy) = p_Y(t,|y|)dy.$$

Note that $r \to \Psi_Y(r)$ and $r \to p_Y(t,r)$ may not be monotone in general.
We first consider the following mild condition on $\Psi_Y$.
\medskip

\noindent
{\bf (A1)}:
There exists a positive function $h$ on $[0, \infty)$  such that  for every $t, \lambda  >0$
$$
{\Psi_Y(\lambda t)}/{\Psi_Y(t)} \le h(\lambda) \quad \text{and} \quad  \int_0^\infty  e^{-r^2/2} r^{d-1}h(r ) dr <\infty.
$$

\medskip

Note that by Lemma \ref{l:phi-property} below, {\bf (A1)}  always holds with
$h(\lambda)=1 \vee \lambda^2$ for every subordinate Brownian motion. Moreover, by
\cite[Lemma 3 and Proposition 11]{G}, {\bf (A1)}  always holds with
$h(\lambda)=24(1 +\lambda^2)$ for rotationally
symmetric unimodal L\'evy process (i.e.,  $r \to p_Y(t,r)$ is decreasing for all $t>0$).

Recall that
$$e^{-|z|^2}=(4\pi)^{-d/2} \int_{\R^d}e^{i\xi \cdot z}e^{-|\xi|^2/4}d \xi.$$
Using this and \eqref{e:Psi} we have for $\lambda>0$
\begin{align}
&\E_0[e^{-\lambda |Y_t|^2}]=(4\pi)^{-d/2} \int_{\R^d}\E_0[ e^{i \sqrt \lambda \xi \cdot Y_t}]e^{-|\xi|^2/4}d \xi \nonumber\\
&=(4\pi)^{-d/2} \int_{\R^d} e^{-t\Psi_Y( \sqrt \lambda|\xi|)} e^{-|\xi|^2/4}d \xi. \label{e:PP1}
\end{align}
Thus
\begin{align}
 \E_0 [e^{-\lambda|Y_t|^2} - e^{-2\lambda|Y_t|^2}]
=(4\pi)^{-d/2} \int_{\R^d} (e^{-t\Psi_Y( \sqrt \lambda|\xi|)}-e^{-t\Psi_Y( \sqrt {2\lambda}|\xi|)}) e^{-|\xi|^2/4}d\xi. \label{e:PP0}
\end{align}
For $t, \lambda>0$, let
$$
g_t(\lambda):=\int_0^\infty (e^{-t\Psi_Y( \sqrt \lambda r)}-e^{-t\Psi_Y( \sqrt {2 \lambda} r)}) e^{-r^2/4} r^{d-1} dr,
$$
which is positive by \eqref{e:PP0}.
\begin{lemma}\label{l:pub1} Suppose that {\bf (A1)} holds. Then there exists a constant $N=N(h,d)$ such that for every $t, v  >0$
$$
g_t( v^{-1}) \le  N t \Psi_Y( v^{-1/2}).
$$
\end{lemma}

\begin{proof}
By {\bf (A1)} we have
\begin{align*}
&\frac{1}{t \Psi_Y( v^{-1/2})}
\le
\frac{\Psi_Y( \sqrt {2} v^{-1/2} r )+\Psi_Y( v^{-1/2} r)}{\Psi_Y( v^{-1/2})}\frac{1}{t|\Psi_Y( \sqrt {2} v^{-1/2} r )- \Psi_Y( v^{-1/2} r)|}\\
\le&
(h(\sqrt {2} r )+h(r)) \frac{1}{t|\Psi_Y( \sqrt {2} v^{-1/2} r )- \Psi_Y( v^{-1/2} r)|}.
\end{align*}
Thus using the inequality $|e^{-a}-e^{-b}| \le |a-b|$, $a, b >0$
\begin{align*}
&\frac{g_t( v^{-1})}{t \Psi_Y( v^{-1/2})}\\
\le&
\int_0^\infty \frac{|e^{-t\Psi_Y( v^{-1/2} r)}-e^{-t\Psi_Y( \sqrt {2} v^{-1/2} r )}|}{t|\Psi_Y( \sqrt {2} v^{-1/2} r )- \Psi_Y( v^{-1/2} r)|}  e^{-r^2/4} r^{d-1}(h(\sqrt {2} r )+h(r))  dr\\
\le&
\int_0^\infty  e^{-r^2/4} r^{d-1}(h(\sqrt {2} r )+h(r))  dr < \infty.
\end{align*}
Therefore the lemma is proved.
\end{proof}

Recall that $\bP_0( Y_t \in dy)=p_Y(t, |y|)dy$.
We now consider the following mild condition on $p_Y(t, r)$.

\medskip

\noindent
{\bf (A2):}
For every $T \in (0, \infty]$, there exists a constant  $c=c(T) >0$  such that for every $t \in (0, T)$
\begin{align}\label{e:pup}
p_Y(t, r)  \le c p_Y(t, s)  \qquad \forall r \ge s  \ge 0.
\end{align}

\medskip

Obviously \eqref{e:pup} always holds on all $t>0$  for rotationally
symmetric unimodal L\'evy process.

\begin{thm} \label{t:hku}
Suppose that
$Y$ is a  rotationally
symmetric L\'evy process with L\'evy exponent
$\Psi_Y(\xi)$ satisfying
{\bf (A1)}.
Assume that
$
\bP(Y_t \in dy) = p_Y(t,|y|)dy$ and {\bf (A2)} holds.
Then  for every $T>0$, there exists a constant $N=N(T, c, d, h)>0$ such that
 $$
 p_Y(t, r) \,\le\, N\,  t\,  r^{-d} \Psi_Y( r^{-1}), \quad (t, r) \in (0, T] \times [0, \infty).
$$
\end{thm}
\begin{proof}
Fix $t \in (0, T]$. For $r \ge 0$ define $f_t(r)=r^{d/2}p_Y(t, r^{1/2}).$
By {\bf (A2)}, for $r \ge 0$,
\begin{align}
&\bP_0(  \sqrt{r/2} <|Y_t| <\sqrt{r}) =\int_{\sqrt{r/2} <|y| <\sqrt{r}} p_Y(t, |y|)dy \nonumber\\
&\ge |B(0,1)| (1-2^{-d/2}) c^{-1} r^{d/2} p_Y(t, r^{1/2})=|B(0,1)| (1-2^{-d/2}) c^{-1}f_t(r). \label{e:PP5}
\end{align}
Denoting $\cL f_t(\lambda)$ the Laplace transform of $f_t$, we have
\begin{align}
&\cL f_t(\lambda) \le N\int_0^\infty \bP_0(  \sqrt{r/2} <|Y_t| <\sqrt{r}) e^{-\lambda r}dr=
N\E_0 \int_{|Y_t|^2}^{2|Y_t|^2}e^{-\lambda r}dr \nonumber\\
&=N\lambda^{-1} \E_0 [e^{-\lambda|Y_t|^2} - e^{-2\lambda|Y_t|^2}]=N\lambda^{-1} g_t(\lambda), \quad  \lambda > 0
\label{e:PP6}
\end{align}
from \eqref{e:PP0}.

Using {\bf (A2)} again,
we get that, for any $v>0$
\begin{align*}
\cL f_t(v^{-1})=&\int^\infty_0e^{-av^{-1} } f_t(a)\, da= v \int^\infty_0e^{-s}f_t\left(sv\right)ds\\
\ge&v \int^{1}_{1/2} e^{-s}f_t\left(sv\right)ds= v \int^{1}_{1/2} e^{-s} s^{d/2}v^{d/2} p_Y\left(t,s^{1/2}v^{1/2}\right)ds \\
\ge&  c^{-1} v 2^{-d/2}  v^{d/2}  p_Y\left(t,v^{1/2}\right) \int^{1}_{1/2} e^{-s}ds = c^{-1}2^{-d/2} v f_t\left(v \right)\left(\int^{1}_{1/2} e^{-s}ds\right).
\end{align*}
Thus
\begin{align}
f_t\left(v\right)\le c 2^{d/2}\frac{v^{-1}\cL f_t(v^{-1})}{e^{-1/2}-e^{-1}}.
\label{e:PP7}
\end{align}

 Now combining  \eqref{e:PP6} and \eqref{e:PP7} with Lemma \ref{l:pub1} we
 conclude
% that $r >0$
 $$
 p_Y(t, r) = r^{-d} f_t(r^2) \le N r^{-d-2}\cL f_t(r^{-2}) \le N r^{-d}g_t(r^{-2}) \le N  t  r^{-d} \Psi_Y( r^{-1}).
$$
\end{proof}

\mysection{Bernstein functions and subordinate Brownian motion}\label{s:2}

%First, we  introduce few properties of  Bernstein functions.

Let $S=(S_t: t\ge 0)$ be a subordinator, that is, an increasing
L\'evy process taking values in $[0,\infty)$ with $S_0=0$. A
subordinator $S$ is completely characterized by its Laplace exponent
$\phi$ via
$$
\E[\exp(-\lambda S_t)]=\exp(-t \phi(\lambda))\, ,\quad  \lambda > 0.
$$
The Laplace exponent $\phi$ can be written in the form (cf. \cite[p.
72]{Ber})
\begin{eqnarray}
\label{01311}
\phi(\lambda)=b\lambda +\int_0^{\infty}(1-e^{-\lambda t})\,
\mu(dt)\, .
\end{eqnarray}
Here $b \ge 0$, and $\mu$ is a $\sigma$-finite measure on
$(0,\infty)$ satisfying
$$
\int_0^{\infty} (t\wedge 1)\, \mu(dt)< \infty\, .
$$
We call the constant $b$  the drift and $\mu$  the L\'evy measure
of the subordinator $S$.

 A smooth function $g : (0, \infty) \to [0, \infty)$ is called a Bernstein function if
 $$
 (-1)^n D^ng \le 0, \quad \forall n\in \bN.
 $$
 It is well known that a nonnegative function $\phi$ on $(0,
\infty)$ is the Laplace exponent of a subordinator if and only if it
is a Bernstein function with $\phi(0+)=0$ (see, for instance, Chapter 3 of \cite{SSV}). By the concavity, for any  Bernstein function $\phi$,
\begin{equation}\label{e:Berall}
\phi(\lambda t)\le \lambda\phi(t) \qquad \text{ for all } \lambda \ge 1, t >0\, ,
\end{equation}
implying
\begin{equation}\label{e:uv}
\frac{\phi(v)}{v}\le \frac{\phi(u)}{u}\, ,\quad 0<u\le v\, .
\end{equation}

Clearly \eqref{e:Berall} implies the following.

\begin{lemma}\label{l:phi-property}
Let $\phi$ be a Bernstein function. Then for all  $\lambda, t>0$,
$1 \wedge  \lambda\le {\phi(\lambda t)}/{\phi(t)} \le 1 \vee
 \lambda$.
\end{lemma}

The following will be used in section \ref{section application} to control the deterministic part of equation (\ref{eqn 0}).

\begin{lemma}
                                              \label{13131}
For each nonnegative integer $n$, there is a constant $N(n)$ such that for every Bernstein function with the drift $b=0$,
\begin{eqnarray}
\label{013122} \frac{\lambda^n |D^n \phi(\lambda)|}{\phi(\lambda)}  \leq N(n), \quad \forall \lambda.
\end{eqnarray}
\end{lemma}
\begin{proof}
The statement is trivial if $n=0$. So let $n \geq 1$. Due to \eqref{01311},
$$
|D^n \phi(\lambda)| = \int_0^\infty t^n e^{-\lambda t} \mu(dt).
$$
Use $t^n e^{-t} \leq N(1-e^{-t})$ to conclude
\begin{eqnarray*}
\lambda^n |D^n \phi(\lambda)| \leq \int_0^\infty (\lambda t)^n e^{-\lambda t} \mu(dt) \leq N \int_0^\infty (1- e^{-\lambda t}) \mu(dt).
\end{eqnarray*}
This obviously leads to \eqref{013122}.
\end{proof}

Throughout this article,
we  assume that
 $\phi$ is a Bernstein functions with the drift $b=0$ and  $\phi(1)=1$.
 Thus
\begin{eqnarray*}
\phi(\lambda)=\int_0^{\infty}(1-e^{-\lambda t})\,
\mu(dt).
\end{eqnarray*}

Let
$d\ge 1$  and $W:=(W_t:t\ge 0)$ be a  $d$-dimensional Brownian motion   with $W_0=0$. Then
$$
\E\left[e^{i\xi\cdot W_t}\right]=
e^{-t|\xi|^2}, \qquad \forall \xi \in \R^d, \,\, t>0
$$
and
$W$ has the transition density
$$
q(t,x,y)=q_d(t,x,y)=(4\pi t)^{-d/2} e^{-\frac{|x-y|^2}{4t}}\, ,\quad x,y\in \R^d, \ t>0\, .
$$

Let $X=(X_t: t\ge 0)$  denote
the subordinate Brownian motion defined by $X_t:=W_{S_t}$. Then
$X_t$ has the characteristic exponent  $\Psi (x)= \phi (|x|^2)$ and  has the  transition density
\begin{equation}\label{e:psbmc}
p(t,x)=p_d(t,x):= \int_{\bR^d} e^{i\xi \cdot x} e^{-t \phi(|\xi|^2)} d\xi.
\end{equation}
For $t\ge 0$, let $\eta_t$ be the distribution of $S_t$. That is, for any Borel set
$A\subset [0,\infty)$, $\eta_t(A)=\bP(S_t\in A)$.
Then we have
\begin{equation}\label{e:psbm}
p(t,x)=p_d(t,x)=\int_{(0, \infty)}(4\pi s)^{-d/2}\exp\left(-\frac{|x|^2}{4s}\right)\, \eta_t(ds)
\end{equation}
(see \cite[Section 13.3.1]{KSV3}). Thus $p(t,x)$ is smooth in $x$.

The L\'evy measure $\Pi$ of $X$ is given by
(see e.g.~\cite[pp. 197--198]{Sat})
$$
\Pi(A)=\int_A \int_0^{\infty}p(t,x)\, \mu(dt)\, dx =\int_A J(x)\,
dx\, ,\quad A\subset {\mathbb R}^d\, ,
$$
where
\begin{equation}\label{ksv-jumping function}
J(x):= \int_0^{\infty}p(t,x)\, \mu(dt)
\end{equation}
is the L\'evy density of $X$. Define the function $j:(0,\infty)\to
(0,\infty)$ as
\begin{equation}\label{ksv-function j measure}
j(r)=j_d(r):=  \int_0^{\infty} (4\pi)^{-d/2} t^{-d/2} \exp\left(-\frac{r^2}{4t}\right)\,
\mu(dt)\, , \quad r>0.
\end{equation}
Then
$
J(x)=j(|x|)$ and
\begin{equation}\label{e:psi}
\Psi(\xi)=\int_{\R^d}(1-\cos(\xi\cdot y))j(|y|)dy.
\end{equation}
Note that the function $r\mapsto
j(r)$ is strictly positive, continuous and decreasing on $(0, \infty)$.

The next lemma is  an extension of \cite[Lemma 3.1]{KSV8}.

\begin{lemma}\label{l:j-g-upper}
 There exists a constant $N>0$ depending only on $d$ such that
    \begin{equation*}\label{e:j-up}
    j(r)\le  N\,r^{-d}\phi(r^{-2})\, ,\qquad \forall r > 0.
    \end{equation*}
\end{lemma}
\begin{proof}
By Lemma \ref{l:phi-property},  {\bf (A1)} holds  with $h(\lambda)=1 \vee \lambda^2$, and {\bf (A2)} holds with $c=1$ since $r \to p(t,r)$ is decreasing.  Thus
by Theorem \ref{t:hku}, we have
\begin{equation}\label{e:hkne1}
p(t,r)\le N t r^{-d}\phi(r^{-2})    \quad \forall t, r>0
\end{equation}
where $N>0$ depends only on $d$.
The lemma now follows from \eqref{e:hkne1} and \eqref{e:ig}.
Indeed,
by \eqref{e:ig} and Section 4.1 in \cite{Sk} that for $f \in C^2_0(\R^d\setminus \{0\})$ (the set of $C^2$-functions on $\bR^d\setminus \{0\}$ with compact support), we have
\begin{equation}\label{e:hkne2}
\lim_{t \to 0}\frac{1}{t} \int_{\R^d} p(t, |y|)f(y)dy=\phi(\Delta) f(0)
= \int_{\R^d} j(|y|)f(y)dy.
\end{equation}
 We fix $r>0$ and  choose a $f \in C^2_0(\bR^d\setminus \{0\})$ such that
$f=1$ on $B(0, r) \setminus B(0, r/2)$ and $f=0$ on $B(0, 2r)^c \cup  B(0, r/4)$.
Note that since $s \to j(s)$ is decreasing,
we have
\begin{align*}
&r^d j(r) \le \frac{d 2^d}{2^d-1} \int_{r/2}^r j(s)s^{d-1} ds \le N \int_{B(0, r) \setminus B(0, r/2)} j(|y|)
dy\\
&\le N \int_{\R^d} j(|y|) f(y)
dy
\end{align*}
where $N>0$ depends only on $d$.
Thus by \eqref{e:hkne1} and \eqref{e:hkne2}, we see that
\begin{align*}
&r^d j(r) \le N \lim_{t \to 0}\frac{1}{t} \int_{\R^d} p(t, |y|)f(y)
dy\le N  \int_{\R^d} |y|^{-d}\phi(|y|^{-2}) f(y)
dy\\
&\le N \int_{B(0, 2r) \setminus B(0, r/4)} |y|^{-d}\phi(|y|^{-2}) dy
\le N  \int_{B(0, 2r) \setminus B(0, r/4)} |y|^{-d} dy  \phi(4r^{-2})\\
&\le N  \int^{2r}_{r/4} r^{-1} dr  \phi(4r^{-2})  \le N  \phi(4r^{-2})\end{align*}
where $N>0$ depends only on $d$.
Now the lemma follows immediately by \eqref{e:Berall}.
\end{proof}

For $a>0$, we define $\phi^a(\lambda)= \phi(\lambda a^{-2})/\phi( a^{-2})$. Then $\phi^aâ$ is again a Bernstein function satisfying $\phi^a(1)=1$. We will use $\mu^a(dt)$  to denote the L\'evy measure
of $\phi^a$ and $S^a=(S^a_t)_{t\ge 0}$ to denote a subordinator with Laplace exponent
$\phi^a$.

Assume that  $S^a=(S^a)_{t\ge 0}$ is  independent of the Brownian motion $W$. Let $X^a=(X^a_t)_{t\ge 0}$ be defined by $X^a_t:=W_{S^a_t}$. Then $X^a$ is a rotationally invariant L\'evy process with characteristic exponent
\begin{equation}\label{e:psi1}
\Psi^a(\xi)=\phi^a(|\xi|^2)=\frac{\phi(a^{-2}|\xi|^2)}{\phi(a^{-2})}=\frac{\Psi(a^{-1}\xi)}{\phi(a^{-2})}\, , \quad \xi\in \R^d\, .
\end{equation}
This shows that $\{ X_t^a-X^a_0\}_{t\ge 0}$ is identical in law to the process $\{a^{-1}(X_{t/\phi(a^{-2})}-X_0)\}_{t\ge 0}$.
$X^1$ is simply the process $X$.

Since, by  \eqref{e:psi} and \eqref{e:psi1},
\begin{equation}\label{e:psi2}
\Psi^a(\xi)=\frac{1}{\phi(a^{-2})}\int_{\R^d}(1-\cos(a^{-1}\xi\cdot y))j(|y|)dy=\frac{a^d}{\phi(a^{-2})}\int_{\R^d}(1-\cos(\xi\cdot z))j(a|z|)dz,
\end{equation}
the L\'evy measure of $X^a$ has the density $J^a(x)=j^a(|x|)$, where $j^a$ is given by
\begin{eqnarray}
j^a(r):= a^d \phi(a^{-2})^{-1} j(ar)\,  \label{e:jumping-function-a}.
\end{eqnarray}

We  use $p^a(t, x, y)=p^a(t, x-y)$ to denote the transition density of $X^a$.
Recall that
the process $\{a^{-1} (X_{t/\phi(a^{-2})}-X_0):t\ge 0\}$ has the same law as
$\{X^a_t-X^a_0: t\ge 0\}$. In terms of transition densities, this can be written as
\begin{equation*}\label{e:scaling4densbm}
p^a(t, x, y)=a^{d}p(\frac{t}{\phi(a^{-2})}, ax, ay), \qquad (t, x, y)\in (0, \infty)\times\R^d\times \R^d.
\end{equation*}
Thus
\begin{equation}\label{e:scaling4densbm1}
p(t, x)=p^1(t, x)=a^{-d}p^a(t\phi(a^{-2}), a^{-1}x), \qquad (t, x)\in (0, \infty)\times\R^d.
\end{equation}
Denote
$$
a_t:=\frac{1}{\sqrt{\phi^{-1} (t^{-1})}}.
$$
From \eqref{e:scaling4densbm1},
we see that
\begin{equation*}\label{e:scaling4densbm2}
p(t, x)=(a_t)^{-d}p^{a_t}(1, (a_t)^{-1}x), \qquad (t, x)\in (0, \infty)\times\R^d.
\end{equation*}

Let $\beta>0$.  For appropriate functions $f=f(x)$, define
$$
T_{t}f(x) := (p(t, \cdot) \ast f(\cdot))(x) = \int_{\bR^d} p (t,x-y)f(y)dy, \quad t>0,
$$
$$
\phi(\Delta)^{\beta} f:=-\phi(-\Delta)^{\beta}f := \cF^{-1} (\phi(|\xi|^2)^{\beta} \cF(f))(x), \quad t>0.
$$
In particular, if  $\beta=1$ and $f\in C_b^2(\R^d)$  then we have
\begin{eqnarray}
&\phi(\Delta) f(x)=-\phi(-\Delta) f(x)=\int_{\R^d}\left(f(x+y)-f(x)-\nabla f(x) \cdot y {\mathbf 1}_{\{|y|\le 1\}}\right)\, j(|y|)\, dy \nonumber\\
& =\lim_{\eps \downarrow 0}\int_{\{y\in
\bR^d: \, |y|>\eps\}} (f(x+y)-f(x)) j(|y|)\, dy  \label{eqn 18.2013}
\end{eqnarray}
(see Section 4.1 in \cite{Sk}).

Recall that  $\phi^{a_t}(\lambda):=\phi(\lambda (a_t)^{-2})/\phi((a_t)^{-2})$.
Since $t\phi(a_t^{-2})=1$, by \eqref{e:psbmc}
\begin{eqnarray}
\phi(\Delta)^{1/2}p(t,\cdot) (x)
&=& \int_{\bR^d}\phi(|\xi|^2)^{1/2} e^{i x \xi}e^{-t\phi(|\xi|^2)} d\xi  \nonumber\\
&=&t^{-1/2}\int_{\bR^d}(\phi(|\xi|^2)/\phi(a_t^{-2}))^{1/2} e^{i  x \xi}e^{-\phi(|\xi|^2)/\phi(a_t^{-2})} d\xi \nonumber\\
&=&t^{-1/2}(a_t)^{-d}\int_{\bR^d}\phi^{a_t}(|\xi|^2)^{1/2} e^{i (a_t)^{-1} x \xi}e^{-\phi^{a_t}(|\xi|^2)} d\xi \nonumber\\
&=&(a_t)^{-d}t^{-1/2}\phi^{a_t}(\Delta)^{1/2}p^{a_t}(1,\cdot) ((a_t)^{-1} x).
\label{e:pscale}
\end{eqnarray}

By \cite[Corollary 3.7 (iii)]{SSV},
$\phi^{a}(\lambda)^{1/2}$ is also a  Bernstein function.
Thus
$\phi^a(\lambda)^{1/2}=\int_0^{\infty}(1-e^{-\lambda t})\,
\wh \mu^a(dt) $ where $\wh \mu^a$ is   the L\'evy measure of $\phi^a(\lambda)^{1/2}$. Let
$$
\wh j^a(r):=\int^{\infty}_0(4\pi t)^{-d/2}e^{-r^2/(4t)}\wh \mu^a(dt), \qquad r, a>0
$$
and $\wh j(r):=\wh j^1(r)$.
Then, by \eqref{e:jumping-function-a}
\begin{eqnarray}
\label{11021}
\wh j^a(r)=  a^d \phi(a^{-2})^{-1/2} \wh j(ar)\,, \qquad r, a>0 . \label{e:whjumping-function-a}
\end{eqnarray}
%So
%we have, for $s,a>0$,
%\begin{eqnarray*}
%\int_{s}^\infty\wh j^a(r) r^{d-1}dr =  a^d \phi(a^{-2})^{-1/2} \int_{s}^\infty  \wh j(ar)r^{d-1}dr
%=  \phi(a^{-2})^{-1/2} \int_{as}^\infty  \wh j(t)t^{d-1}dt
%\end{eqnarray*}
%and
As (\ref{eqn 18.2013}), for every $f \in C_b^2(\R^d)$,
\begin{eqnarray}
\phi^a(\Delta)^{1/2} f(x)&:=&-\phi^a(-\Delta)^{1/2} f(x) \nonumber\\
&=&\int_{\R^d}\left(f(x+y)-f(x)-\nabla f(x) \cdot y {\mathbf 1}_{\{|y|\le r\}}\right)\,\wh  j^a(|y|)\, dy  \nonumber\\
 &=&\lim_{\eps \downarrow 0}\int_{\{y\in \bR^d: \, |y|>\eps\}} (f(x+y)-f(x)) \wh  j^a(|y|)\, dy\label{e:1over2}
\end{eqnarray}

Clearly by Lemma \ref{l:j-g-upper} we have the following. We emphasize that the constant does not depend on neither  $\phi$ nor $a$.

\begin{lemma}\label{l:j-g-uppera}
There exists a constant $N>1$ depending only on $d$ such that for any $a>0$ and  $x\neq 0$
    \begin{equation*}\label{e:j-up}
    \wh  j^a (r)\le N r^{-d}(\phi^a(r^{-2}))^{1/2}\, ,\qquad \forall r > 0.
    \end{equation*}
\end{lemma}

\vspace{5mm}

Recall conditions {\bf{(H1)}} and {\bf{(H2)}}:

\noindent
{\bf (H1):}
There exist constants $0<\delta_1 \le \delta_2 <1$ and $a_1, a_2>0$  such that
\begin{equation*}
a_1\lambda^{\delta_1} \phi(t) \le \phi(\lambda t) \le a_2 \lambda^{\delta_2}\phi(t)  \quad \lambda \ge 1, t \ge 1\, .
\end{equation*}
\noindent
{\bf (H2):}
There exist constants $0<\delta_3\leq1$ and $a_3>0$  such that
\begin{equation*}
 \phi(\lambda t) \le a_3 \lambda^{\delta_3} \phi(t), \quad \lambda \le 1, t \le 1\, .
\end{equation*}

\vspace{4mm}

By taking $t=1$ in {\bf (H1)} and {\bf (H2)} and using Lemma \ref{l:phi-property}, we get that
if {\bf (H1)} holds then
\begin{equation}\label{e:H1-large}
a_1\lambda^{\delta_1} \le \phi(\lambda ) \le a_2 \lambda^{\delta_2}\, , \quad \lambda \ge 1\, ,
\end{equation}
and, if {\bf (H2)} holds then
\begin{equation}\label{e:H2-small}
 \lambda \le \phi(\lambda ) \le a_3 \lambda^{\delta_3}\, , \quad \lambda \le 1\, .
\end{equation}
Also, if  {\bf (H1)} holds  we have
\begin{equation}\label{e:H1-scaled}
a_1\lambda^{\delta_1} \phi^a(t) \le \phi^a(\lambda t) \le a_2 \lambda^{\delta_2} \phi^a(t)\, ,\quad \lambda \ge 1, t \ge a^2\,.
\end{equation}
%PK: we don't use this.
%By Lemma \ref{l:phi-property}, if  {\bf (H2)} holds we have
%\begin{equation*}\label{e:H2-scaled}
%\lambda^{-1} \phi^a(t) \le \phi^a(\lambda^{-1} t) \le a_3 \lambda^{-\delta_3} \phi^a(t)\, , \quad \lambda \ge 1, t \le a^2\, .
%\end{equation*}
%Thus
%if $a\le 1$, by taking $t=1$ in \eqref{e:H1-scaled}, we get
%\begin{equation}\label{e:la2}
%a_1\lambda^{\delta_1} \le \phi^a(\lambda ) \le a_2 \lambda^{\delta_2}\, ,\quad \lambda \ge 1.
%\end{equation}
Thus,
by taking $t=1$ in \eqref{e:H1-scaled}, if {\bf (H1)} holds and $a\le 1$ we get
$$
a_1\lambda^{\delta_1} \le \phi^a(\lambda ) \le a_2 \lambda^{\delta_2}\, ,\quad \lambda \ge 1.
$$
Thus, if {\bf (H1)} holds
\begin{equation}\label{e:la2}
a_1(T^{-2} \wedge 1)\lambda^{\delta_1} \le \phi^a(\lambda ) \le \frac{a_2}{\phi(T^{-2}) \wedge 1} \lambda^{\delta_2}\, ,\quad a \in (0, T], \, \lambda \ge 1.
\end{equation}
 In fact, if $T>1$ and $1 \le a \le T$ then for $\lambda \ge 1$
$$
\phi^a(\lambda )= \frac{\phi(\lambda a^{-2})}{\phi( a^{-2})} \le \frac{\phi(\lambda)}{\phi( a^{-2})} \le \frac{a_2 \lambda^{\delta_2}}{\phi( T^{-2})}
$$
and using Lemma \ref{l:phi-property}
$$
\phi^a(\lambda )= \frac{\phi(\lambda a^{-2})}{\phi( a^{-2})} \ge\frac{\phi(\lambda a^{-2})}{\phi( 1)}=\frac{\phi(\lambda a^{-2})}{\phi(\lambda )}\phi(\lambda ) \ge a^{-2} \phi(\lambda ) \ge T^{-2} \phi(\lambda ) \ge T^{-2}a_1\lambda^{\delta_1}.
$$

 Recall that  $p(t,x)$ is the  transition density of $X_t$.

\begin{corollary} \label{c:ekub}
Suppose {\bf (H1)} holds. Then for each $T>0$ there exists a constant $N=N(T,d,\phi)>0$ such that for $(t, x) \in (0, T] \times \R^d$,
\begin{equation}\label{e:psbm_assup}
p(t, x) \le N\left(\left(\phi^{-1}(t^{-1})\right)^{d/2}   \wedge  t \frac{\phi(|x|^{-2})}{|x|^d} \right). \end{equation}
\end{corollary}
\pf
The corollary follows from Theorem \ref{t:hku} and the first display on page 1073 of \cite{CKK2}.
Also one can see from \eqref{e:psbmc} and \eqref{e:scaling4densbm1} that
\begin{eqnarray*}
&&p(t,  x) =(a_t)^{-d}p^{a_t}(1, a_t^{-1}x)
\le (a_t)^{-d}\int_{\bR^d} e^{-\phi^{a_t}(|\xi|^2)} d\xi \\
&\le&(a_t)^{-d} \int_{|\xi| <1}  d\xi
+(a_t)^{-d}
\int_{|\xi| \ge 1} e^{-\phi^{a_t}(|\xi|^2)} d\xi
\\
&\le& N(a_t)^{-d}\left(1+
\int_{|\xi| \ge 1} e^{-a_1(a_T^{-2} \wedge 1)|\xi|^{2\delta_1}} d\xi \right).
\end{eqnarray*}
where \eqref{e:la2} is used in the last inequality. \qed

\begin{remark}{\rm
If
there exist constants $0<\delta_3\le \delta_4 < 1$ and $a_3, a_4>0$  such that
\begin{equation}\label{e:H}
a_4\lambda^{\delta_4} \phi(t) \le \phi(\lambda t) \le a_3 \lambda^{\delta_3} \phi(t), \quad \lambda \le 1, t \le 1\, ,
\end{equation}
 then, as in \cite{KSV8}
 the subordinate Brownian motion $X$ satisfies conditions (1.4), (1.13) and (1.14) from \cite{CK}.
 Thus, in fact, by \cite{CK}, if  {\bf (H1)} and \eqref{e:H} hold we have the sharp two-sided estimates for all $t>0$
\begin{equation}\label{e:psbm_p}
N^{-1}\left(\left(\phi^{-1}(t^{-1})\right)^{d/2}   \wedge  t \frac{\phi(|x|^{-2})}{|x|^d} \right) \le p(t, x) \le N\left(\left(\phi^{-1}(t^{-1})\right)^{d/2}   \wedge  t \frac{\phi(|x|^{-2})}{|x|^d} \right),
 \end{equation}
 where $N=N(\phi)>1$.

On the other hand, when  $\delta_3=\delta_4 =1$ in \eqref{e:H}, \eqref{e:psbm_p} does not hold and \eqref{e:psbm_assup}
is not sharp. For example, see \cite[(2.2), (2.4) and Theorem 4.1]{CKS2}. }
\end{remark}

For the rest of   this article we  assume that {\bf (H1)} and {\bf (H2)} hold.
Now we further discuss the scaling. If $0<r<1<R$, using \eqref{e:uv}, \eqref{e:H1-large} and \eqref{e:H2-small}, we have
$$
\frac{\phi(R)}{\phi(r)}  \le  \frac{R}{r}
\quad \text{ and }\quad
\frac{\phi(R)}{\phi(r)} \ge \frac{a_1}{a_3} \frac{R^{\delta_1}}{r^{\delta_3}}
\ge \frac{a_1}{a_3}  \left(\frac{R}{r}\right)^{\delta_1 \wedge \delta_3}.
$$
Combining these with {\bf (H1)} and {\bf (H2)} we get
\begin{equation}\label{e:sc1}
\frac{a_1}{a_3} \left(\frac{R}{r}\right)^{\delta_1 \wedge \delta_3} \le \frac{\phi(R)}{\phi(r)} \le \frac{R}{r}, \quad 0<r<R<\infty\, .
\end{equation}
Now applying this to $\phi^a$, we get
\begin{equation}\label{e:sc2}
\frac{a_1}{a_3}\left(\frac{R}{r}\right)^{\delta_1 \wedge \delta_3} \le \frac{\phi^a(R)}{\phi^a(r)} \le  \frac{R}{r}, \quad a>0,\ 0<r<R<\infty\, .
\end{equation}

Next two lemmas will be used several times in this article.

\begin{lemma}
         \label{l:integral-estimates-phi}
Assume {\bf (H1)} and {\bf (H2)}. Then there exists a constant $N=N(\delta_1,\delta_3)$ such that for all $\lambda >0$
    \begin{equation*}
    \int_{\lambda^{-1}}^\infty r^{-1} \phi (r^{-2})\, dr \le  N \phi(\lambda^2). \label{e:ie-2}
    \end{equation*}
\end{lemma}

\begin{proof}
Changing variable $r \to \lambda^{-1}r$, from \eqref{e:sc1} we have
\begin{eqnarray*}
\int_{\lambda^{-1}}^\infty r^{-1} \phi (r^{-2})\, dr
~=~\int_{1}^\infty r^{-1} \phi (\lambda^2 r^{-2})\, dr
&=&\int_{1}^\infty r^{-1} \phi (\lambda^2 r^{-2}) \frac{\phi(\lambda^2)}{\phi(\lambda^2)}\, dr \\
&\leq& \int_{1}^\infty r^{-1 -2 \delta_1 \wedge \delta_3}\, dr \phi(\lambda^2).
\end{eqnarray*}
\end{proof}
\begin{corollary}
\label{cor116}
Assume {\bf (H1)} and {\bf (H2)}. Then there exists a constant $N=N(\delta_1,\delta_3)$ such that for all $\lambda >0$
    \begin{equation*}
    \int_{\lambda^{-1}}^\infty r^{-1} (\phi(r^{-2}))^{1/2}\, dr \le  N (\phi(\lambda^2))^{1/2}, \label{e:ie-2}
    \end{equation*}
\end{corollary}
\begin{proof}
Since $(\phi(\lambda))^{1/2}$ also satisfies conditions {\bf (H1)} and {\bf (H2)} with different $\delta_1,\delta_3>0$, we  get this corollary directly from the previous lemma.
\end{proof}

\mysection{Upper bounds of $|\phi(\Delta)^{n/2}D^{\beta}p(t,x)|$}\label{s:3}

In this section we give upper bounds of  $|\phi(\Delta)^{n/2}D^{\beta}p(t,x)|$ for any  $n=0,1,2,\cdots$ and  multi-index $\beta$.

Recall $a_t:=(\phi^{-1} (t^{-1}))^{-1/2}$ and so $t\phi(a_t^{-2})=1$. Thus
From \eqref{e:scaling4densbm1} and \eqref{e:psbm_assup}, we have for every $t\in (0,T]$,
\begin{eqnarray}
                                      \label{e:psbm_assup1}
p^{a_t}(1, x) &=& (a_t)^{d}p(t, a_t x) \\
&\le& N(T,\phi,d) \left(    1  \wedge  t \frac{\phi(|a_t x|^{-2})}{|x|^d} \right)
= N \left(1  \wedge  \frac{\phi^{a_t}(|x|^{-2})}{|x|^d} \right). \nonumber
\end{eqnarray}

\begin{lemma}
                            \label{le0new}
For any constant $T>0$ there exists a constant $N=N(T,d,\phi)$ so that for every $t \in (0, T]$
\begin{align*}%\label{e:psbm310old}
|\nabla p^{a_t}(1, x)| \leq N |x| \left(      1   \wedge  \frac{\phi^{a_t}(|x|^{-2})}{|x|^{d+2}} \right),
\end{align*}
\begin{align*}%\label{e:psbm32old}
\sum_{i,j} |\partial_{x_i,x_j } p^{a_t}(1,x)|&\le N  |x|^2  \left(  1  \wedge  \frac{\phi^{a_t}(|x|^{-2})}{|x|^{d+4}} \right)+
N \left( 1  \wedge   \frac{\phi^{a_t}(|x|^{-2})}{|x|^{d+2}} \right),
\end{align*}
and
\begin{align*}%\label{e:psbm33old}
\sum_{|\beta| \leq n}  |D^\beta p^{a_t}(1,x)|&\le N \sum_{n-2m \geq 0, m \in \bN \cup \{0\}} |x|^{n-2m}  \left(      1  \wedge  \frac{\phi^{a_t}(|x|^{-2})}{|x|^{d+2(n-m)}} \right).
\end{align*}

\end{lemma}
\begin{proof}
To distinguish the dimension, we denote
$$
p^{a_t}_d(1,x):=\int_{(0, \infty)}(4\pi s)^{-d/2}\exp\left(-\frac{|x|^2}{4s}\right)\, \eta^{a_t}_1(ds).
$$
By \eqref{e:psbm},
\begin{align*}%\label{e:psbm1}
\partial_{x_i} p^{a_t}_d(1,x)&=\int_{(0, \infty)}(4\pi s)^{-d/2}\partial_{x_i} \exp\left(-\frac{|x|^2}{4s}\right)\, \eta^{a_t}_1(ds)\nonumber\\
&=-\frac{x_i}{2} \int_{(0, \infty)}s^{-1}(4\pi s)^{-d/2} \exp\left(-\frac{|x|^2}{4s}\right)\, \eta^{a_t}_1(ds)\\
&=-2\pi x_i p^{a_t}_{d+2}(1,(x, 0,0)),
\end{align*}
\begin{align*}%\label{e:psbm2}
\partial_{x_i,x_i} p^{a_t}_d(1,x)&=4 \pi x_i^2  p^{a_t}_{d+4}(1,(x, 0,0,0,0)) -2\pi  p^{a_t}_{d+2}(1,(x, 0,0)),
\end{align*}
and, for $i \not=j$,
\begin{align*}%\label{e:psbm3}
\partial_{x_i,x_j } p^{a_t}_d(1,x)&=4 \pi x_i x_j  p^{a_t}_{d+4}(1,(x, 0,0,0,0)).
\end{align*}
Thus
\begin{align*}%\label{e:psbm4}
|\nabla p^{a_t}_d(1, x)| \leq 2\pi |x| p^{a_t}_{d+2}(1,(x, 0,0))
\end{align*}
and
\begin{align*}%\label{e:psbm5}
\sum_{i,j} |\partial_{x_i,x_j } p^{a_t}_d(1,x)|&\le 4 d^2 \pi |x|^2  p^{a+t}_{d+4}(1,(x, 0,0,0,0))+2 d \pi  p^{a_t}_{d+2}(1,(x, 0,0)).
\end{align*}
Similarly.
\begin{align*}%\label{e:psbm6}
\sum_{i,j,k} |\partial_{x_i,x_jx_k } p^{a_t}_d(1,x)|&\le  N(d)[|x|^3  p^{a_t}_{d+6}(1,(x,0,0, 0,0,0,0))+|x|p^{a_t}_{d+4}(1,(x,0,0, 0,0))].
\end{align*}
Repeating the product rule of differentiation and applying \eqref{e:psbm_assup1}, we prove the lemma.
\end{proof}

\begin{lemma}
                                         \label{le1}
For any $(t,x)\in (0,T]\times \bR^d$ and multi-index $\beta$,
\begin{align*}
|\phi(\Delta)^{1/2} D^\beta p(t, \cdot)(x)|\leq \,  N(d, T, |\beta|, \phi)\ \left(t^{-1/2}   \left(\phi^{-1}(t^{-1})\right)^{(d+|\beta|)/2}  \wedge \frac{\phi(| x|^{-2})^{1/2}}{ | x|^{d+|\beta|}}\right).
\end{align*}
\end{lemma}

\begin{proof}
 First we prove the lemma when $\beta=0$.
Recall
$$a_t:=\frac{1}{\sqrt{\phi^{-1} (t^{-1})}} \le \frac{1}{\sqrt{\phi^{-1} (T^{-1})}}.$$
Note that by \eqref{e:psbmc}
\begin{eqnarray*}
&&|\phi^{a_t}(\Delta)^{1/2}p^{a_t}(1,\cdot) (x)|\\
&=&|\int_{\bR^d}\phi^{a_t}(|\xi|^2)^{1/2} e^{i x \xi}e^{-\phi^{a_t}(|\xi|^2)} d\xi | \\
&\le& \int_{\bR^d}\phi^{a_t}(|\xi|^2)^{1/2} e^{-\phi^{a_t}(|\xi|^2)} d\xi \\
&\le& \int_{|\xi| <1}\phi^{a_t}(|\xi|^2)^{1/2}  d\xi
+\left(\sup_{b>0} b^{1/2} e^{-b/2} \right)
\int_{|\xi| \ge 1} e^{-2^{-1}\phi^{a_t}(|\xi|^2)} d\xi
 \\
&\le& N \phi^{a_t}(1)^{1/2}+N
\int_{|\xi| \ge 1} e^{-2^{-1}\phi^{a_t}(|\xi|^2)} d\xi \le N+N
\int_{|\xi| \ge 1} e^{-2^{-1}\phi^{a_t}(|\xi|^2)} d\xi.
\end{eqnarray*}
Thus it is uniformly bounded by \eqref{e:sc2}. By \eqref{e:1over2},
\begin{align}
&|\phi^{a_t}(\Delta)^{1/2} p^{a_t}(1, \cdot)(x)| \nonumber\\
&=
|\lim_{\eps \downarrow 0}\int_{\{y\in
\bR^d: \, |y|>\eps\}} (p^{a_t}(1,x+y)-p^{a_t}(1,x)) \wh  j^{a_t}(|y|)\, dy| \nonumber\\
&\le
|p^{a_t}(1,x)|
\int_{\{y\in
\bR^d: \, |y|>|x|/2\}}  \wh  j^{a_t}(|y|)\, dy+\int_{\{y\in
\bR^d: \, |y|>|x|/2\}} |p^{a_t}(1,x+y)| \wh  j^{a_t}(|y|)\, dy \nonumber\\
&\quad +
|
\lim_{\eps \downarrow 0}\int_{\{y\in
\bR^d: \,|x|/2 > |y|>\eps\}} \int_0^1 | \nabla p^{a_t}(1,x+sy)| ds |y|\wh  j^{a_t}(|y|)\, dy \nonumber\\
&:=p^{a_t}(1,x) \times I+II+III. \nonumber
\end{align}
Since from \eqref{11021}
\begin{eqnarray*}
\int_{s}^\infty\wh j^a(r) r^{d-1}dr =  a^d \phi(a^{-2})^{-1/2} \int_{s}^\infty  \wh j(ar)r^{d-1}dr
=  \phi(a^{-2})^{-1/2} \int_{as}^\infty  \wh j(t)t^{d-1}dt,
\end{eqnarray*}
we get
\begin{eqnarray*}
\int_{|x|}^\infty\wh j^{a_t}(r) r^{d-1}dr
&=& (a_t)^d \phi((a_t)^{-2})^{-1/2} \int_{|x|}^\infty  \wh j(a_t r)r^{d-1}dr\\
&=&  \phi((a_t)^{-2})^{-1/2} \int_{|x|a_t}^\infty  \wh j(s)s^{d-1}ds\\
&=& \phi^{a_t}(|x|^{-2})^{1/2} \phi((a_t)^{-2}|x|^{-2})^{-1/2}\int_{|x|a_t}^\infty  \wh j(s)s^{d-1}ds.
\end{eqnarray*}
In addition to this, applying Lemma  \ref{l:j-g-uppera} and Corollary \ref{cor116}, we see
$$
\phi((a_t)^{-2}|x|^{-2})^{-1/2}\int_{|x|a_t}^\infty  \wh j(s)s^{d-1}ds  \le N.
$$
Combining these with
$$p^{a_t}(1, x) \le N(T) \left(      1\wedge   \frac{\phi^{a_t}(|x|^{-2})}{|x|^d} \right),$$
we get
\begin{align*}
p^{a_t}(1,x) \times I &\le
N(T) \phi^{a_t}(|x|^{-2})^{1/2}  \left(      1\wedge   \frac{\phi^{a_t}(|x|^{-2})}{|x|^d} \right) \le
N(T)  \frac{\phi^{a_t}(|x|^{-2})^{3/2} }{|x|^d}.
\end{align*}

Using the fact that $r \to j^{a_t}(r)$ is decreasing,
\begin{align*}
II &\le  \wh  j^{a_t}(|x|/2) \int_{\{y\in
\bR^d: \, |y|>|x|/2\}} p^{a_t}(1,x+y)\, dy  \\
&\le  \wh  j^{a_t}(|x|/2) \int_{
\bR^d} p^{a_t}(1,x+y)\, dy \le \wh  j^{a_t}(|x|/2).
\end{align*}

Finally, from Lemma \ref{le0new} we get
\begin{eqnarray*}
&&III =|
\lim_{\eps \downarrow 0}\int_{\{y\in
\bR^d: \,|x|/2 > |y|>\eps\}} \int_0^1 | \nabla p^{a_t}(1,x+sy)| ds |y|\wh  j^{a_t}(|y|)\, dy\\
&&\leq N\int_{|y| < |x|/2}  \int_0^1   \frac{\phi^{a_t}(|x+sy|^{-2})}{|x+sy|^{d+1}} ds
|y| \wh  j^{a_t}(|y|)dy\\
\\
&&\leq N\int_{|y| < |x|/2}  \int_0^1   \frac{\phi^{a_t}((|x|-s|y|)^{-2})}{(|x|-s|y|)^{d+1}} ds
|y| \wh  j^{a_t}(|y|)dy
\\
&&\leq N\int_{|y| < |x|/2}  \int_0^1   \frac{\phi^{a_t}(4|x|^{-2})}{|x|^{d+1}} ds
|y| \wh  j^{a_t}(|y|)dy\\
\\
&&\leq N\frac{\phi^{a_t}(4|x|^{-2})}{|x|^{d+1}} \int_{|y| < |x|/2}
|y| \wh  j^{a_t}(|y|)dy.
\end{eqnarray*}

By Lemma \ref{l:j-g-uppera},
\begin{eqnarray*}
&&\int_{|y| < |x|/2}
|y| \wh  j^{a_t}(|y|)dy \le N \int_{|y| < |x|/2}      |y|^{-d+1}(\phi^{a_t}(|y|^{-2}))^{1/2}dy\\
&& \le N \int_0^{|x|}     (\phi^{a_t}(r^{-2}))^{1/2}dr.
\end{eqnarray*}
Since $|\phi^{a_t}(\Delta)^{1/2} p^{a_t}(1, \cdot)(x)|$ is bounded for $x$, we may assume that $|x| \geq 1$.
So from the monotone property of $\phi^{a_t}(r^{-2})$ and \eqref{e:la2},  we get
\begin{eqnarray*}
&&\int_0^{|x|}     (\phi^{a_t}(r^{-2}))^{1/2}dr\\
&=& \int_0^{1}     (\phi^{a_t}(r^{-2}))^{1/2}dr + \int_{ 1}^{ |x|}     (\phi^{a_t}(r^{-2}))^{1/2}dr\\
&\leq& N \left(\int_0^{ 1}    r^{-\delta_2}dr + \int_{1}^{|x|}(\phi^{a_t}(r^{-2}))^{1/2}dr \right)\\
&\leq& N \left(|x|^{1-\delta_2}+  |x| \phi^{a_t}(1)^{1/2} \right).
\end{eqnarray*}
%On the other hand, if $a_t > 1$, then similarly from \eqref{e:H1-large} we get
%\begin{eqnarray*}
%\int_0^{|x|}     (\phi^{a_t}(r^{-2}))^{1/2}dr
%&=& \int_0^{a_t^{-1}}     (\phi^{a_t}(r^{-2}))^{1/2}dr+\int_{a_t^{-1}}^{|x|}     (\phi^{a_t}(r^{-2}))^{1/2}dr\\
%&=&  \int_0^{a_t^{-1} }     (\frac{ \phi(r^{-2} a_t^{-2})}{ \phi(a_t^{-2})})^{1/2} dr+\int_{a_t^{-1} }^{|x|}     (\phi^{a_t}(r^{-2}))^{1/2}dr \\
%&\leq& N  \left( (\frac{a_t^{-2\delta_2}}{ \phi(a_t^{-2})})^{1/2} \int_0^{a_t^{-1}}    r^{-\delta_2}dr+\int_{a_t^{-1}}^{|x|}     (\phi^{a_t}(r^{-2}))^{1/2}dr \right)\\
%&\leq& N \left((\frac{1}{\phi(a_T^{-2})})^{1/2} |x|^{1-\delta_2}+    |x| \phi^{a_t}(a_t^2)^{1/2} \right)\\
%&\leq& N (\frac{1}{\phi(a_T^{-2})})^{1/2} \left(|x|^{1-\delta_2}+    |x|  \right)\leq NT^{1/2}(|x|^{1-\delta_2}+|x|).
%\end{eqnarray*}
Thus
$$
|\phi^{a_t}(\Delta)^{1/2} p^{a_t}(1, \cdot)(x)| \le N (1 \wedge \frac{\phi^{a_t}(|x|^{-2})^{1/2}}{|x|^{d}}
).
$$

Now applying \eqref{e:pscale} and using the fact that $t\phi(a_t^{-2})=1$ and $\phi^a(\lambda)=\phi(\lambda a^{-2})/\phi(a^{-2})$, we get
\begin{eqnarray}
&&|\phi(\Delta)^{1/2}p(t,\cdot) (x)|
=(a_t)^{-d}t^{-1/2}|\phi^{a_t}(\Delta)^{1/2}p^{a_t}(1,\cdot) ((a_t)^{-1} x)|\nonumber\\
&&\le \,  N\,  t^{-1/2}\,((a_t)^{-d} \wedge  \frac{\phi^{a_t}(|(a_t)^{-1} x|^{-2})^{1/2}}{| x|^{d}})\nonumber\\
&&=\,  N\, t^{-1/2}\, ((a_t)^{-d}\wedge \frac{\phi(| x|^{-2})^{1/2}}{\phi( (a_t)^{-2})^{1/2} | x|^{d}})\nonumber\\
&&=\,  N\,  \, ( t^{-1/2}(a_t)^{-d} \wedge\frac{\phi(| x|^{-2})^{1/2}}{(t\phi( (a_t)^{-2}))^{1/2} | x|^{d}})\, =\, N\,  \, ( t^{-1/2}(a_t)^{-d} \wedge\frac{\phi(| x|^{-2})^{1/2}}{ | x|^{d}}). \nonumber
\label{e:pscalew}
\end{eqnarray}

 The case $|\beta| =1$ is proved similarly. First, one can  check that
$|\phi^{a_t}(\Delta)^{1/2} D^{\beta}p^{a_t}(1, \cdot)(x)|$
is  uniformly bounded. Note
\begin{eqnarray}
&&|\phi^{a_t}(\Delta)^{1/2} D^\beta p^{a_t}(1, \cdot)(x)| \nonumber\\
&=&
|\lim_{\eps \downarrow 0}\int_{\{y\in
\bR^d: \, |y|>\eps\}} (D^\beta p^{a_t}(1,x+y)-D^\beta p^{a_t}(1,x)) \wh  j^{a_t}(|y|)\, dy| \nonumber\\
&\le&
|D^\beta p^{a_t}(1,x)|
\int_{\{y\in
\bR^d: \, |y|>|x|/2\}}  \wh  j^{a_t}(|y|)\, dy \nonumber\\
&&+|\int_{\{y\in \bR^d: \, |y|>|x|/2\}} D^\beta p^{a_t}(1,x+y) \wh  j^{a_t}(|y|)\, dy| \nonumber\\
&& +|
\lim_{\eps \downarrow 0}\int_{\{y\in
\bR^d: \,|x|/2 > |y|>\eps\}} \int_0^1 | \nabla D^\beta p^{a_t}(1,x+sy)| ds |y|\wh  j^{a_t}(|y|)\, dy \nonumber\\
&:=& |D^{\beta}p^{a_t}(1,x)| \times I+II+III. \label{eqn 12.18.2}
\end{eqnarray}
Since  $I$ and $III$ can be estimated similarly as  in  the case $|\beta|=0$, we only pay attention to the estimation of $II$. We use integration by parts and get
\begin{eqnarray*}
II &\le&   \int_{ |y|=|x|/2} |\wh  j^{a_t}(|x|/2) p^{a_t}(1,x+y)| dS\\
 &&+     \int_{\{y\in \bR^d: \, |y|>|x|/2\}}  \frac{d}{dr} \wh j^{a_t}(|y|)  p^{a_t}(1,x+y)\, dy.
\end{eqnarray*}
We use notation $  \wh j^{a_t}_d(r)$  in place of $\wh j^{a_t}(r)$ to express its dimension. That is,
\begin{equation*}
\wh j^{a_t}_d(r):=\int^{\infty}_0(4\pi t)^{-d/2}e^{-r^2/(4t)}\wh \mu^{a_t}(dt), \qquad r>0.
\end{equation*}
By its definition, we can easily see that $\frac{d}{dr} \wh j^{a_t}_d(r) = -2\pi r  \wh j^{a_t}_{d+2}(r)$.
So from Lemmas  \ref{l:j-g-uppera} and  \ref{le0new} we get
\begin{align*}
II &\le N[ \frac{\phi^{a_t}(|x|^{-2})^{3/2}}{|x|^{d+1}}+  \int_{|y| > |x|/2}  |y|  \wh j^{a_t}_{d+2}(|y|) |p^{a_t}(1,x+y)|  \, dy ]\\
&\le N[ \frac{\phi^{a_t}(|x|^{-2})^{3/2}}{|x|^{d+1}}+  \int_{|y| > |x|/2} \frac{\phi^{a_t}(|y|^{-2})^{1/2}}{|y|^{d+1}} |p^{a_t}(1,x+y)|  \, dy ]\\
&\le N[\frac{\phi^{a_t}(|x|^{-2})^{3/2}}{|x|^{d+1}} +  \frac{\phi^a(|x|^{-2})^{1/2}}{|x|^{d+1}}].
\end{align*}
Therefore we get
$$
|\phi^{a_t}(\Delta)^{1/2} p_{x^i}^{a_t}(1, \cdot)(x)| \le N (1 \wedge \frac{\phi^{a_t}(|x|^{-2})^{1/2}}{|x|^{d+1}}).
$$
By fact that $t\phi(a_t^{-2})=1$ and $\phi^a(\lambda)=\phi(\lambda a^{-2})/\phi(a^{-2})$,
\begin{eqnarray}
\notag &&|\phi(\Delta)^{1/2}p_{x^i}(t,\cdot) (x)|\\
\notag &=& \left|\int_{\bR^d} (-i\xi^i)\phi(|\xi|^2)^{1/2} e^{i x \xi}e^{-t\phi(|\xi|^2)} d\xi \right|\\
\notag &=&t^{-1/2}\left|\int_{\bR^d} (-i\xi^i) (\phi(|\xi|^2)/\phi(a_t^{-2}))^{1/2} e^{i  x \xi}e^{-\phi(|\xi|^2)/\phi(a_t^{-2})} d\xi \right|\\
\notag &=&t^{-1/2}(a_t)^{-d-1}\left|\int_{\bR^d} (-i\xi^i) \phi^{a_t}(|\xi|^2)^{1/2} e^{i (a_t)^{-1} x \xi}e^{-\phi^{a_t}(|\xi|^2)} d\xi \right|\\
\notag &=&(a_t)^{-d-1}t^{-1/2}\left|\phi^{a_t}(\Delta)^{1/2}p_{x^i}^{a_t}(1,\cdot) ((a_t)^{-1} x)\right|\\
\notag &\le& \,  N\,  t^{-1/2} (a_t)^{-d-1} \,\left(1 \wedge  \frac{\phi^{a_t}(|a_t^{-1} x|^{-2})^{1/2}}{|a_t^{-1} x|^{d+1}} \right)\\
\notag &=&\,  N \, t^{-1/2}  (a_t)^{-d-1} \, \left( 1 \wedge \frac{\phi(| x|^{-2})^{1/2}}{\phi( (a_t)^{-2})^{1/2} | a_t^{-1}x|^{d+1}}\right)\\
\notag &=&\,  N\,  t^{-1/2}(a_t)^{-d-1} \, \left( 1 \wedge  t^{1/2} a_t^{d+1} \frac{\phi(| x|^{-2})^{1/2}}{(t\phi( (a_t)^{-2}))^{1/2} | x|^{d+1}} \right)\\
\label{12201}&=&\, N\,  \, \left( t^{-1/2}(a_t)^{-d-1} \wedge\frac{\phi(| x|^{-2})^{1/2}}{ | x|^{d+1}}\right).
\end{eqnarray}
Finally, we consider the case $|\beta| \geq 2$. Introduce $I, II$ and $III$ as in (\ref{eqn 12.18.2}). $I$ and $III$ can be estimated as in the case $|\beta|=0$.
Also, $II$ can be estimated by doing the integration by parts $|\beta|$-times. For instance, if $|\beta|=2$,
\begin{eqnarray*}
II &\le&   \int_{ |y|=|x|/2} |\wh  j^{a_t}(|x|/2) p_{x^i}^{a_t}(1,x+y)| dS
+\int_{ |y|=|x|/2} |\frac{d}{dr} \wh  j^{a_t}(|x|/2) p^{a_t}(1,x+y)| dS \\
&&+\int_{\{y\in \bR^d: \, |y|>|x|/2\}}  \frac{d^2}{dr^2} \wh j^{a_t}(|y|)  p^{a_t}(1,x+y)\, dy,
\end{eqnarray*}
where $dS$ is the surface measure on $\{y\in \bR^d: |y|=|x|/2\}$.
By its definition, we easily see that
$$
\frac{d}{dr} \wh j^{a_t}_d(r) = -2\pi r  \wh j^{a_t}_{d+2}(r), \quad
\frac{d^2}{dr^2} \wh j^{a_t}_d(r) = -2\pi  \wh j^{a_t}_{d+2}(r) +(2\pi)^2r^2 \wh j^{a_t}_{d+4}(r).
$$
So from Lemma  \ref{l:j-g-uppera} and Lemma \ref{le0new} and we get
\begin{align*}
II &\le  N\left( \frac{\phi^{a_t}(|x|^{-2})^{3/2}}{|x|^{d+2}}+  \int_{|y| > |x|/2} |y|^2 \wh j_{d+4}^{a_t}(|y|)|p^{a_t}(1,x+y)|  \, dy \right) \\
&\le  N\left(\frac{\phi^{a_t}(|x|^{-2})^{3/2}}{|x|^{d+2}}+ \int_{|y|>|x|/2}  |y|^{-(d+2)} \phi^{a_t}(|y|^{-2})^{1/2}|p^{a_t}(1,x+y)|  \, dy \right)\\
&\le  N\left(\frac{\phi^{a_t}(|x|^{-2})^{3/2}}{|x|^{d+2}} +  \frac{\phi^a(|x|^{-2})^{1/2}}{|x|^{d+2}}\right).
\end{align*}
Therefore we have
$$
|\phi^{a_t}(\Delta)^{1/2} p_{x^ix^j}^{a_t}(1, \cdot)(x)| \le N\left(1 \wedge \frac{\phi^{a_t}(|x|^{-2})^{1/2}}{|x|^{d+2}}\right)
$$
and we are done  by the scaling  as in \eqref{12201}.
\end{proof}

We generalize Lemma \ref{le1} as follows.

\begin{lemma}
                                      \label{le3}
For any $n\in \bN$ and multi-index $\beta$, there exists a constant $N=N(d,\phi, T,|\beta|,n)>0$  such that
\begin{eqnarray}
                                             \label{1220}
|\phi(\Delta)^{n/2} D^\beta p(t, \cdot)(x)|
 \leq   N \left(t^{-n/2}   (\phi^{-1}(t^{-1}))^{(d+|\beta|)/2}  \wedge t^{-(n-1)/2}\frac{\phi(| x|^{-2})^{1/2}}{ | x|^{d+|\beta|}}\right).
\end{eqnarray}
\end{lemma}
\begin{proof}
We use the induction. Due to the previous lemma, the statement is true if $n=1$.
Assume that the statement is true for $n-1$.  We put
\begin{eqnarray*}
&&|\phi^{a_t}(\Delta)^{n/2} D^\beta p^{a_t}(1, \cdot)(x)|\\
&=&
\lim_{\eps \downarrow 0}\int_{\{y\in
\bR^d: \, |y|>\eps\}} (\phi(\Delta)^{(n-1)/2}D^\beta p^{a_t}(1,x+y)-\phi(\Delta)^{(n-1)/2}D^\beta p^{a_t}(1,x)) \wh  j^{a_t}(|y|)\, dy|\\
&\le&
|\phi(\Delta)^{(n-1)/2}p^{a_t}(1,x)|
\int_{\{y\in
\bR^d: \, |y|>|x|/2\}}  \wh  j^{a_t}(|y|)\, dy \\
&& +\int_{\{y\in
\bR^d: \, |y|>|x|/2\}} |\phi(\Delta)^{(n-1)/2} p^{a_t}(1,x+y)| \wh  j^{a_t}(|y|)\, dy\\
&& +|\lim_{\eps \downarrow 0}\int_{\{y\in
\bR^d: \,|x|/2 > |y|>\eps\}} \int_0^1 |  \phi(\Delta)^{1/2} \nabla p^{a_t}(1,x+sy)| ds |y|\wh  j^{a_t}(|y|)\, dy \\
&:=&|\phi(\Delta)^{1/2}p^{a_t}(1,x)|\times I+II+III
\end{eqnarray*}
and
follow the proof of Lemma \ref{le1} with the result \eqref{1220} for $n-1$.  Below, we provide details only for $II$.
Let $|\beta|=0$. Then since  $r \to j^{a_t}(r)$ is decreasing,
\begin{align*}
II &\le  \wh  j^{a_t}(|x|/2) \int_{\{y\in
\bR^d: \, |y|>|x|/2\}} \phi(\Delta)^{(n-1)/2} p^{a_t}(1,x+y)\, dy\\
&\le  \wh  j^{a_t}(|x|/2) \int_{
\bR^d} \phi(\Delta)^{(n-1)/2} p^{a_t}(1,x+y)\, dy \le N \wh  j^{a_t}(|x|/2) \leq N \frac{\phi^{a_t}(|x|^{-2})^{1/2} }{|x|^d}.
\end{align*}
If $|\beta| > 0$ then as in Lemma \ref{le1} we use integration by parts $|\beta|$-times and get
$$
II \leq  N\frac{\phi^{a_t}(|x|^{-2})^{1/2} }{|x|^{d+|\beta|}}.
$$
Therefore, since $|\phi^{a_t}(\Delta)^{n/2} D^\beta p^{a_t}(1, \cdot)(x)|$ is uniformly bounded, we have
$$
|\phi^{a_t}(\Delta)^{n/2} D^\beta p^{a_t}(1, \cdot)(x)|
\leq N\left(\frac{\phi^{a_t}(|x|^{-2})^{1/2} }{|x|^{d+|\beta|}} \wedge 1\right)
$$
and the lemma is proved by the scaling  as in \eqref{12201}.
\end{proof}

\mysection{Proof of Theorem \ref{main theorem}}\label{s:proof}

Let $f\in C^{\infty}_0(\bR^{d+1},H)$. For each  $a\in \bR$ denote
$$
u_a(t,x):=\cG_a(t,x):=  [\int_{a}^t |\phi(\Delta)^{1/2}T_{t-s}f(s,\cdot)(x)|_H^2 ds]^{1/2},
$$
$u(t,x):=u_0(t,x)$ and $\cG(t,x):=\cG_0(t,x)$.

Here is a version of Theorem \ref{main theorem} for $p=2$.
\begin{lemma}
\label{2-1}
 For
 any  $\infty \geq \beta  \geq \alpha \geq -\infty$ and $\beta \geq a$,
\begin{eqnarray}
\label{4023} \|u_a\|^2_{L_2( [\alpha,\beta] \times \bR^d)} \leq
N\||f|_H\|^2_{L_2([a,\beta] \times \bR^d)},
\end{eqnarray}
where $N=N(d)$.
\end{lemma}

\begin{proof}
By the continuity of $f$, the range of $f$ belongs to a
separable subspace of $H$. Thus by using  a countable orthonormal
basis of this subspace and the Fourier transform one easily finds

\begin{eqnarray*}
&& \|u_a\|^2_{L_2([\alpha,\beta] \times \bR^{d})}\\
&=& (2\pi)^d \int_{\bR^d} \int_{\alpha}^\beta  \int_{a}^t |\cF\{\phi(\Delta)^{1/2}p (t-s, \cdot)\}(\xi)|^2 \, |\cF(f)(s,\xi)|^2_H ds dt d\xi\\
 &\leq& (2\pi)^d \int_{\bR^d} \int_{a}^\beta \int_{\alpha}^\beta I_{0 \leq t-s} \phi(|\xi|^2)e^{-2(t-s)(\phi(|\xi|^2))} dt|\cF(f)(s,\xi)|^2_H  ds d\xi.
\end{eqnarray*}
Changing $t-s \to t$, we find that the last term above is equal to
\begin{eqnarray*}
&& (2\pi)^d \int_{\bR^d} \int_{a}^\beta \int_{\alpha-s}^{\beta-s} I_{0 \leq t}
\phi(|\xi|^2)e^{-2t(\phi(|\xi|^2))} dt |\cF(f)(s,\xi)|^2_H
ds d\xi\\
&\leq& (2\pi)^d \int_{\bR^d} \int_{a}^\beta \int_{0}^{\infty}
\phi(|\xi|^2)e^{-2t(\phi(|\xi|^2))} dt |\cF(f)(s,\xi)|^2_H
ds d\xi.
\end{eqnarray*}
Since $\int_{0}^{\infty}\phi(|\xi|^2)e^{-2t(\phi(|\xi|^2))} dt =1/2$, we have
\begin{eqnarray*}
\|u_a\|^2_{L_2( [\alpha,\beta] \times \bR^{d} )} \leq N \int_{a}^\beta
\int_{\bR^d} |\hat{f}(s,\xi)|^2_H ~d\xi ds.
\end{eqnarray*}
The last expression is equal to the right-hand side of (\ref{4023}),
and therefore  the lemma is proved.
\end{proof}

For $c>0$ and $(r,z) \in \bR^{d+1}$, we denote
$$
 B_c(z)= \{ y \in \bR^d : |z-y| <c\}, \quad
  \hat{B}_c(z) = \prod_{i=1}^d(z^i-c/2,z^i+c/2),
  $$
  $$
   I_c(r)= ({r-\phi(c^{-2})^{-1}},\, r+\phi(c^{-2})^{-1}), \quad  Q_c(r,z)=  I_r(c)\times \hat{B}_c(z).
 $$
 Also we denote
 $$
 Q_c(r) =Q_c(r,0), \quad \hat{B}_c= \hat{B}_c(0), \quad B_c= B_c(0).
 $$
For a measurable function $h$ on $\bR^d$, define the
maximal functions
$$
\bM_x h(x) := \sup_{r>0} \frac{1}{|B_r(x)|} \int_{B_r(x)} |h(y)| dy,
$$
$$
\cM_x h(x) := \sup_{\hat B_r(z) \ni x} \frac{1}{|\hat B_r(z)|} \int_{\hat B_r(z)} |h(y)| dy.
$$
 Similarly,
for a measurable function $h=h(t)$ on $\bR$,
$$
\bM_t h(t) := \sup_{r>0} \frac{1}{2r} \int_{-r}^r |h(t+s)|\, ds.
$$
Also for a function $h=h(t,x)$, we set
$$
\bM_x h(t,x) := \bM_x(h(t,\cdot))(x),
\quad \cM_x h(t,x) := \cM_x(h(t,\cdot))(x),
$$
$$
 \bM_t\bM_xh(t,x) = \bM_t(\bM_xh(\cdot,x))(t), \quad \cM_x\bM_t\bM_x h(t,x) =\cM_x(\bM_t\bM_x h(t,\cdot))(x).
 $$

\vspace{3mm}

Since we  have estimates of $p(t,x)$ only for  $t\leq T$,  we introduce the following functions.
Denote $\hat{p}(t,x)=p(t,x)$ if $t\in [0,T]$, and $\hat{p}(t,x)=0$ otherwise. Define $\hat{T}_tf(x)= \int_{\bR^d} \hat p(t,y) f(x-y)dy$,
and for every $a\in \bR$
$$
\hat u_a(t,x):=\hat \cG_a(t,x):=  \begin{cases} \left[\int_{a}^t |\phi(\Delta)^{1/2}\hat T_{t-s}f(s,\cdot)(x)|_H^2 ds \right]^{1/2} &:\,t  \geq a \\
\left[\int_{a}^{2a-t} |\phi(\Delta)^{1/2}\hat T_{2a-t-s}f(s,\cdot)(x)|_H^2 ds \right]^{1/2}  &: \, t < a.
\end{cases}
$$
We use  $\hat u(t,x)$ and $\hat \cG(t,x)$ in place of $\hat u_0(t,x)$ and $\hat \cG_0(t,x)$ respectively.  Obviously,
Lemmas \ref{le1} and \ref{le3} hold with $\hat p(t,x)$ instead of $p(t,x)$ (for all $t$).
Moreover,
\begin{equation}
               \label{eqn 12.20.1}
\hat u_a (t,x) =\hat u_a(2a-t,x) \quad \forall \,t\in \bR, \quad \quad  \hat{u}_a(t,x)\leq u_a(t,x) \quad  \text{if}\,\,\, t\geq a,
\end{equation}
\begin{equation}
                      \label{eqn 12.20.2}
 \hat{u}_a(t,x)=u_a(t,x) \quad \text{if}\,\,\, t\in [a, T+a].
\end{equation}

\begin{lemma}
                                    \label{co1}
Assume that  the support of $f$ belongs to $\bR \times B_{3dc}$. Then for any $c>0$ and $(t,x) \in Q_c(r)$
\begin{eqnarray*}
 \int_{Q_c(r)} |\hat u_a(s,y)|^2 ~dsdy \leq N [|r- a| + \phi(c^{-2})^{-1}]c^d  \bM_t \bM_x |f|_H^2 (t,x),
\end{eqnarray*}
where $N$ depends only on $d$.
\end{lemma}

\begin{proof}
Fix $(t,x)\in Q_c(r)$. Using  (\ref{eqn 12.20.1}) and  Lemma \ref{2-1}, we get
\begin{eqnarray*}
&&\int_{Q_c(r)} |\hat u_a(s,y)|^2 ~dsdy\\
&\leq& \int_{(r-\phi(c^{-2})^{-1}) \wedge a}^{ (r+\phi(c^{-2})^{-1}) \vee a} \int_{\bR^d}|\hat u_a(s,y)|^2  dyds \\
&=& \int_{\bR^d}\left[\int_{(r-\phi(c^{-2})^{-1}) \wedge a}^{ a} |\hat u_a(2a-s,y)|^2  ds
+ \int_a^{ (r+\phi(c^{-2})^{-1}) \vee a} |\hat u_a(s,y)|^2  ds\right] dy \\
&=& \int_{\bR^d}\left[\int_{a}^{2a-[(r-\phi(c^{-2})^{-1}) \wedge a] } |\hat u_a(s,y)|^2  ds + \int_a^{ (r+\phi(c^{-2})^{-1}) \vee a} |\hat u_a(s,y)|^2 ds\right] dy \\
&\leq& N\int_{B_{3dc}}\left[\int_{a}^{2a-[(r-\phi(c^{-2})^{-1}) \wedge a]}  |f(s,y)|_H^2ds +  \int_{a}^{(r+ \phi(c^{-2})^{-1})\vee a }  |f(s,y)|_H^2ds\right]dy.
\end{eqnarray*}
Since  $|x-y| \leq |x|+|y| \leq   4dc$ for any $(t,x) \in Q_c(r)$ and $y
\in B_{3dc}$, the last term above is less than or equal to constant times of
$$
\int_{|x-y|\leq 4dc}\left[\int_{a}^{2a-[(r-\phi(c^{-2})^{-1}) \wedge a]}  |f(s,y)|_H^2ds +  \int_{a}^{(r+ \phi(c^{-2})^{-1})\vee a }  |f(s,y)|_H^2ds\right]dy
$$
\begin{eqnarray*}
 &\leq& Nc^d \left[\int_{a}^{2a-[(r-\phi(c^{-2})^{-1}) \wedge a]} \bM_x|f(s,x)|_H^2ds +  \int_{a}^{(r+ \phi(c^{-2})^{-1})\vee a }\bM_x |f(s,x)|_H^2ds\right]\\
 &\leq& N [|r- a| + \phi(c^{-2})^{-1}]c^d  \bM_t \bM_x |f|_H^2 (t,x).
\end{eqnarray*}
In order to explain the last inequality above, we denote
$$
\int_{a}^{2a-[(r-\phi(c^{-2})^{-1}) \wedge a]} \bM_x|f(s,x)|_H^2ds +  \int_{a}^{(r+ \phi(c^{-2})^{-1})\vee a }\bM_x |f(s,x)|_H^2ds := I + J.
$$
First we estimate $I$. $I=0$ if  $ r - \phi(c^{-2})^{-1} \geq a$.  So assume $r - \phi(c^{-2})^{-1}  < a$.

If $a \leq t \leq 2a - (r-\phi(c^{-2})^{-1})$, then we can easily get
$$
I \leq [|r- a| + \phi(c^{-2})^{-1}]  \bM_t \bM_x |f|_H^2 (t,x).
$$
If $t > 2a-(r-\phi(c^{-2})^{-1})$ and $t \geq a$, then
\begin{eqnarray*}
I \leq \int_{a}^{t+(t-a)} \bM_x|f(s,x)|_H^2ds &\leq& 2(t-a) \bM_t \bM_x |f|_H^2 (t,x) \\
&\leq& 2(r+\phi(c^{-2})^{-1} -a) \bM_t \bM_x |f|_H^2 (t,x).
\end{eqnarray*}
Finally, if $ t <a$, then
\begin{eqnarray*}
I &\leq& \int_{t}^{2a -(r-\phi(c^{-2})^{-1})} \bM_x|f(s,x)|_H^2ds
\leq \int_{t+t-[2a -(r-\phi(c^{-2})^{-1})]}^{2a -(r-\phi(c^{-2})^{-1})} \bM_x|f(s,x)|_H^2ds \\
&\leq& 2([2a -(r-\phi(c^{-2})^{-1})]-t) \bM_t \bM_x |f|_H^2 (t,x)\\
&\leq& 4 [a -(r-\phi(c^{-2})^{-1})] \bM_t \bM_x |f|_H^2 (t,x).
\end{eqnarray*}
The estimation of $J$ is similar. Therefore, the lemma is proved.
\end{proof}

We will use the following version of integration by parts: if  $0 \leq \varepsilon \leq R \leq \infty$,
and $F$ and $G$ are smooth enough then (see \cite{Kr01})
\begin{eqnarray}
\notag \int_{R \geq |z| \geq \varepsilon} F(z) G(|z|)~dz
= - \int_\varepsilon^R G'(\rho)(\int_{|z| \leq \rho} F(z) dz) d\rho \\
+ G(R) \int_{|z| \leq R} F(z)dz - G(\varepsilon) \int_{|z| \leq
\varepsilon} F(z)dz. \label{4026}
\end{eqnarray}

We generalize Lemma \ref{co1} as follows.

 \begin{lemma}
                                        \label{2-3}
For any $(t,x) \in
 Q_c(r)$
\begin{eqnarray*}
 \int_{Q_c(r)} |\hat u_a(s,y)|^2 ~dsdy \leq N [|r- a| + \phi(c^{-2})^{-1}]c^d \bM_t \bM_x |f|_H^2 (t,x),
\end{eqnarray*}
 where $N=N(d,T,\phi)$.
\end{lemma}

\begin{proof}

Take  $\zeta \in C_0^\infty (\bR^d)$ such that $\zeta =1$ in
$B_{2dc}$, $\zeta=0$ outside of $B_{3dc}$, and $0 \leq \zeta \leq 1$. Set $\cA = \zeta f$ and
$\cB = (1-\zeta)f$. By Minkowski's inequality, $\hat \cG_a f \leq \hat \cG_a \cA+
\hat \cG_a \cB$. Since $\hat \cG_a \cA$ can be estimated by Lemma \ref{co1},
 assume that $f(t,x)=0$ for $x \in B_{2dc}$.  Let $e_1 := (1,0,\ldots,0)$ and $s\geq \mu\geq a$. Then since $\phi(\Delta)^{1/2}\hat p(t,x)$ is rotationally invariant with respect to $x$, we have
\begin{eqnarray}
&&|\phi(\Delta)^{1/2} \hat T_{ s-\mu}f(\mu, \cdot)(y)|_H \nonumber \\
&=&|\int_{\bR^d} \phi(\Delta)^{1/2} \hat p (s-\mu,|z|e_1)f( \mu ,y-z) ~dz|_H   \nonumber\\
&=&|\int_0^\infty \phi(\Delta)^{1/2} \hat p_{x^1} (s-\mu ,\rho e_1)   \int_{|z| \leq \rho} f(\mu ,y-z) ~d\rho|_H. \label{eqn 12.21.1}
\end{eqnarray}
For the second equality above, \eqref{4026} is used with $G(|z|)=\phi(\Delta)^{1/2} \hat p (s-\mu,|z|e_1)$ and $F(z)=f( \mu ,y-z)$.
Observe that if $y \in \hat B_c$ then for $ x \in (-c/2 , c/2)^d$
\begin{eqnarray*}
|x-y| \leq 2dc, \quad B_\rho(y) \subset B_{2dc+\rho}(x).
\end{eqnarray*}
Moreover, if $|z| \leq c$, then $|y-z| \leq  2dc$ and $f(\mu,y-z)=0$.
Thus by Corollary \ref{cor116} and  Lemma \ref{le1},
\begin{eqnarray}
|\phi(\Delta)^{1/2} \hat T_{ s- \mu}f(\mu, \cdot)(y)|_H
&\leq&  N\int_c^\infty \frac{\phi(\rho^{-2})^{1/2} }{\rho^{d+1}}\int_{|z| \leq \rho} |f|_H(\mu,y-z)~dz\,d\rho  \nonumber\\
&=&  N\int_c^\infty \frac{\phi(\rho^{-2})^{1/2} }{\rho^{d+1}}\int_{B_{\rho}(y)} |f|_H(\mu,z)~dz\,d\rho \nonumber \\
&\leq&  N\int_c^\infty \frac{\phi(\rho^{-2})^{1/2} }{\rho^{d+1}} \int_{B_{2dc+\rho}(x)} |f|_H(\mu,z)~dz\,d\rho  \nonumber\\
&\leq&  N\int_c^\infty \frac{\phi(\rho^{-2})^{1/2} }{\rho^{d+1}} \int_{B_{2d\rho+\rho}(x)} |f|_H(\mu,z)~dz\,d\rho  \nonumber \\
&\leq&  N\bM_x|f|_H(\mu,x)\int_c^\infty \frac{\phi(\rho^{-2})^{1/2} }{\rho} ~d\rho \nonumber\\
 &\leq& N \phi(c^{-2})^{1/2} \bM_x|f|_H(\mu,x). \nonumber
\end{eqnarray}
By Jensen's
inequality $(\bM_x|f|_H)^2 \leq \bM_x |f|_H^2$, and
therefore,  we get for any $s \geq a$ and $y \in \hat B(c)$
\begin{eqnarray*}
|\hat u_a(s,y)|^2
&\leq& N \phi(c^{-2})  \int_{ a}^{s  } \bM_x |f|_H^2(\mu,x) d\mu.
\end{eqnarray*}
So if $r+\phi(c^{-2})^{-1} \geq s \geq a$, then we have
\begin{eqnarray*}
|\hat u_a(s,y)|^2
&\leq&  N \phi(c^{-2})  [r+\phi(c^{-2})^{-1}-(a \wedge (r-\phi(c^{-2})^{-1}))]  \bM_t\bM_x |f|_H^2(t,x)\\
&\leq&  N \phi(c^{-2})[|r-a| +\phi(c^{-2})^{-1}]  \bM_t\bM_x |f|_H^2(t,x).
\end{eqnarray*}
If $ r-\phi(c^{-2})^{-1} \leq s <a$, then we get
\begin{eqnarray*}
|\hat u_a(s,y)|^2
=|\hat u_a(2a-s,y)|^2 &\leq& N \phi(c^{-2})  \int_{ a}^{2a-s  } \bM_x |f|_H^2(\mu,x) d\mu\\
&\leq&  N \phi(c^{-2})[|r-a| +\phi(c^{-2})^{-1}]  \bM_t\bM_x |f|_H^2(t,x).
\end{eqnarray*}

Therefore, we get for any $(t,x) \in Q_c(r)$
\begin{eqnarray*}
 \int_{Q_{c}(r)} |\hat u_a(s,y)|^2 ~dsdy
 \leq N  [|r -a|+\phi(c^{-2})^{-1}] c^d \cdot \bM_t \bM_x |f|_H^2 (t,x).
\end{eqnarray*}
The lemma is proved.
\end{proof}

 \begin{lemma}
                                        \label{2-4}
Assume $2\phi(c^{-2})^{-1} <r$. Then for any $(t,x) \in Q_c(r)$,
\begin{align}
 &\int_{Q_c(r)} \int_{r-\phi(c^{-2})^{-1}}^{r+\phi(c^{-2})^{-1}} |\hat u(s_1,y)-\hat u(s_2,y)|^2 ~ds_1 ds_2 dy  \nonumber\\
&\leq N  \phi(c^{-2})^{-2} c^d   \left[\bM_t \bM_x |f|_H^2 (t,x) +\cM_x \bM_t  \bM_x |f|_H^2 (t,x)+\cM_x\bM_t\bM_x |f|^2_H(t-T,x)\right], \label{eqn 12.20.4}
\end{align}
 where $N=N(d,T,\phi)$.
\end{lemma}
\begin{proof}
Due to the symmetry, to estimate the left term of (\ref{eqn 12.20.4}), we only consider the case $s_1 >s_2$.
Since $r-\phi(c^{-2})^{-1} > \phi(c^{-2})^{-1} > 0$,  we have $s_2> 0$ for any  $(s_2,y) \in Q_c(r)$. Observe that
by Minkowski's inequality
\begin{eqnarray*}
&&|\hat u(s_1,y)- \hat u(s_2,y)|^2\\
&=&  |(\int_{0}^{s_1} |\phi(\Delta)^{1/2}\hat T_{s_1-\kappa}  f(\kappa,\cdot)(y)|_H^2 d\kappa)^{1/2} -
(\int_{0}^{s_2}  | \phi(\Delta)^{1/2}\hat T_{s_2-\kappa}  f(\kappa,\cdot)(y)|_H^2 d\kappa)^{1/2}|^2 \\
&\leq&  \int_{s_2}^{s_1} |\phi(\Delta)^{1/2}\hat T_{s_1-\kappa}  f(\kappa,\cdot)(y) |_H^2 d\kappa\\
 &&+ \int_{0}^{s_2}  |\phi(\Delta)^{1/2}\hat T_{s_1-\kappa}f(\kappa,\cdot)(y)-\phi(\Delta)^{1/2} \hat T_{s_2-\kappa}  f(\kappa,\cdot)(y) |_H^2 d\kappa \\
 &:=&I^0(s_1,s_2,y) + J^0(s_1,s_2,y).
\end{eqnarray*}

 One can easily estimate $I^0$ using  Lemma \ref{2-3} with $a=s_2$ and $|r-a|\leq 2 \phi(c^{-2})^{-1}$, and thus  we only  need to show that
$$
\int_{Q_c(r)} \int_{r-\phi(c^{-2})^{-1}}^{r+\phi(c^{-2})^{-1}} J^0(s_1,s_2,y)  \, ds_1 ds_2 dy
$$
is less than or equal to the right hand side of (\ref{eqn 12.20.4}).  We divide $J^0$ into two parts :
\begin{eqnarray*}
I:=\int_{0}^{r-2\phi(c^{-2})^{-1}} |\phi(\Delta)^{1/2}\hat T_{s_1-\kappa}f(\kappa,\cdot)(y) - \phi(\Delta)^{1/2}\hat T_{s_2-\kappa}  f(\kappa,\cdot)(y) |_H^2 d\kappa
\end{eqnarray*}
\begin{eqnarray*}
J:=\int_{r-2\phi(c^{-2})^{-1}}^{s_2} |\phi(\Delta)^{1/2}\hat T_{s_1-\kappa}f(\kappa,\cdot)(y) - \phi(\Delta)^{1/2}\hat T_{s_2-\kappa}  f(\kappa,\cdot)(y) |_H^2 d\kappa.
\end{eqnarray*}

Note that since $s_1\geq s_2\geq r-\phi(c^{-2})^{-1}$,  we have  $\eta s_1+(1-\eta)s_2-\kappa \geq \phi(c^{-2})^{-1}$ for  any $\eta \in [0,1]$ and $\kappa\in [0,r-2\phi(c^{-2})^{-1}]$.

If $s_1-\kappa >T $ and $s_2 -\kappa \leq T$ then
$$
|\phi(\Delta)^{1/2}\hat T_{s_1-\kappa}-\phi(\Delta)^{1/2}\hat T_{s_2-\kappa}  f(\kappa,\cdot)(y)|_H^2
=|\phi(\Delta)^{1/2}\hat T_{s_2-\kappa}  f(\kappa,\cdot)(y)|_H^2.
$$
Otherwise, using $\frac{\partial}{\partial t}\phi(\Delta)^{1/2} T_t f(x) = \phi(\Delta)^{3/2} T_t f(x)$, we get
\begin{eqnarray*}
 &&|\phi(\Delta)^{1/2}\hat T_{s_1-\kappa}f(\kappa,\cdot)(y)-\phi(\Delta)^{1/2}\hat T_{s_2-\kappa}  f(\kappa,\cdot)(y)|_H^2\\
 &\leq& (s_1-s_2)^2  \int_0^1 | \phi(\Delta)^{3/2}\hat T_{\eta s_1+(1-\eta)s_2-\kappa}  f(\kappa,\cdot)(y)|_H^2  d\eta.
\end{eqnarray*}
  Therefore,
\begin{eqnarray*}
I&\leq& \int_{s_2-T}^{s_1-T}  |\phi(\Delta)^{1/2}\hat T_{s_2-\kappa}  f(\kappa,\cdot)(y)|_H^2 d\kappa\\
&+&
(s_1-s_2)^2 \int_{0}^{r-2\phi(c^{-2})^{-1}}  \int_0^1 | \phi(\Delta)^{3/2}\hat T_{\eta s_1+(1-\eta)s_2-\kappa}  f(\kappa,\cdot)(y)|_H^2  d\eta d\kappa.
\end{eqnarray*}

Denote $\bar{s}=\bar{s}(\eta)=\eta s_1+(1-\eta)s_2$. As in (\ref{eqn 12.21.1}),
\begin{eqnarray*}
&&|\phi(\Delta)^{3/2} \hat T_{ \bar s-\kappa}f(\kappa, \cdot)(y)|_H\\
&=&|\int_0^\infty \phi(\Delta)^{3/2} \hat p_{x^1} (\bar s-\kappa ,\rho e_1)   \int_{|z| \leq \rho} f(\kappa ,y-z)dz d\rho|_H\\
&\leq& \bM_x |f|_H(\kappa,y) \int_0^{\infty} |\phi(\Delta)^{3/2} \hat p_{x^1} (\bar s-\kappa,\rho e_1)|   \rho^d \, d\rho.
\end{eqnarray*}
'Note that $\bar s-\kappa\geq \phi(c^{-2})^{-1}$ and so $\phi^{-1}((\bar{s}-\kappa)^{-1})\leq c^{-2}$.  By Lemma \ref{le3} and Corollary \ref{cor116},
\begin{eqnarray*}
&&\int_0^{\infty} |\phi(\Delta)^{3/2} \hat p_{x^1} (\bar{s}-\kappa,\rho e_1)|   \rho^d \,d\rho\\
& =&\int^c_0 |\phi(\Delta)^{3/2} \hat p_{x^1} (\bar{s} -\kappa,\rho e_1)| \rho^d \,d\rho
+ \int^{\infty}_c |\phi(\Delta)^{3/2} \hat p_{x^1} (\bar{s} -\kappa,\rho e_1)|  \rho^d\,d\rho \\
&\leq& N \int_0^{c }   (\bar{s}-\kappa)^{-3/2} \phi^{-1}((\bar{s}-\kappa)^{-1})^{(d+1)/2} \rho^dd\rho  + N (\bar{s}-\kappa)^{-1} \int^{\infty}_c \frac{\phi(\rho^{-2})^{1/2} }{\rho^{d+1}}\rho^d \,d\rho\\
&\leq& N  (\bar{s}-\kappa)^{-1}\left[ \phi(c^{-2})^{1/2} c^{-(d+1)} \int_0^{c } \rho^d   d\rho+  \phi(c^{-2})^{1/2}\right]
\leq N (\bar{s}-\kappa)^{-1}\phi(c^{-2})^{1/2}.
\end{eqnarray*}

Similarly, one can check
$$
|\phi(\Delta)^{1/2}\hat T_{s_2-\kappa}  f(\kappa,\cdot)(y)|_H \leq  N \phi(c^{-2})^{1/2} \bM_x |f|_H(\kappa,y).
$$
Therefore, remembering $|s_1-s_2| \leq 2\phi(c^{-2})^{-1}$, we get
\begin{eqnarray}
I&\leq& N\phi(c^{-2})^{-1} \int_0^{r-2\phi(c^{-2})^{-1}}(r-\phi(c^{-2})^{-1}-\kappa)^{-2} \bM_x |f|^2_H(\kappa,y) d\kappa  \label{eqn 12.24.7}\\
&& + N\phi(c^{-2}) \int_{s_2-T}^{s_1-T} \bM_x |f|^2_H(\kappa,y)d\kappa. \nonumber
\end{eqnarray}
Note that $\bM_x |f|^2_H(\kappa,y)$ in (\ref{eqn 12.24.7}) can be replaced by $I_{0<\kappa< r-2\phi(c^{-2})}$ times of it. Thus  by integration by parts,
\begin{eqnarray*}
&&\phi(c^{-2})^{-1} \int_{0}^{r-2\phi(c^{-2})^{-1}} (r-\phi(c^{-2})^{-1}-\kappa)^{-2}  \bM_x |f|^2_H(\kappa,y)  d\kappa \\
&\leq&  N\phi(c^{-2})^{-1}  \int_{-\infty}^{r-2\phi(c^{-2})^{-1}} (r-\phi(c^{-2})^{-1}-\kappa)^{-3}  \int_{\kappa}^{r+\phi(c^{-2})^{-1}} \bM_x |f|^2_H(\nu,y)  d\nu d\kappa \\
&\leq&  N\phi(c^{-2})^{-1} \bM_t\bM_x |f|^2_H(t,y)  \int_{-\infty}^{r-2\phi(c^{-2})^{-1}} \frac{r+\phi(c^{-2})^{-1}-\kappa }{({r-\phi(c^{-2})^{-1}}-\kappa)^3} d\kappa \\
&=&  N\phi(c^{-2})^{-1} \bM_t\bM_x |f|_H^2(t,y)  \int_{\phi(c^{-2})^{-1}}^\infty \frac{\kappa +2\phi(c^{-2})^{-1}}{\kappa^3} d\kappa \\
&\leq&  N \bM_t\bM_x |f|_H^2(t,y).
\end{eqnarray*}
Also,
\begin{eqnarray*}
\phi(c^{-2}) \int_{s_2-T}^{s_1-T} \bM_x |f|^2_H(\kappa,y)d\kappa &\leq& \phi(c^{-2}) \int_{r-\phi(c^{-2})^{-1}-T}^{r+\phi(c^{-2})^{-1}-T} \bM_x |f|^2_H(\kappa,y)d\kappa\\
&\leq& 2\bM_t\bM_x|f|^2_H(t-T,y).
\end{eqnarray*}
Therefore,
$$
I \leq N [\bM_t\bM_x |f|^2_H(t,y)+\bM_t\bM_x |f|^2_H(t-T,y)],
$$
where $N$ depends only on $d, T, a_i,\delta_i$ $(i=1,2,3)$, and  this certainly implies
\begin{eqnarray*}
\notag && \int_{\hat B_c} \int_{r-\phi(c^{-2})^{-1}}^{r+\phi(c^{-2})^{-1}}\int_{r-\phi(c^{-2})^{-1}}^{r+\phi(c^{-2})^{-1}}  I~ ds_1 ds_2 dy \\
&\leq& N  \phi(c^{-2})^{-2}  c^d  [\cM_x\bM_t \bM_x |f|_H^2 (t,x) + \cM_x\bM_t\bM_x |f|^2_H(t-T,x)].
\end{eqnarray*}
It only remains to estimate $J$. Since $s_1 \geq s_2$,
\begin{eqnarray*}
&& \int_{r-2\phi(c^{-2})^{-1}}^{s_2} |\phi(\Delta)^{1/2}\hat T_{s_1-\kappa}f(\kappa,\cdot)(x) - \phi(\Delta)^{1/2}\hat T_{s_2-\kappa}   f(\kappa,\cdot)(y) |_H^2 d\kappa\\
%&\leq& 2\int_{r-2\phi(c^{-2})^{-1}}^{s_2}  |\phi(\Delta)^{1/2}\hat T_{s_1-\kappa}f(\kappa,\cdot)(x)|_H^2  + 2|\phi(\Delta)^{1/2}\hat T_{s_2-\kappa}   f(\kappa,\cdot)(y) |_H^2 %d\kappa\\
&\leq& 2\int_{r-2\phi(c^{-2})^{-1}}^{s_1} |\phi(\Delta)^{1/2}\hat  T_{s_1-\kappa}f(\kappa,\cdot)(x)|_H^2 d\kappa\\
&&+ 2\int_{r-2\phi(c^{-2})^{-1}}^{s_2}|\phi(\Delta)^{1/2} \hat T_{s_2-\kappa}   f(\kappa,\cdot)(y) |_H^2 d\kappa.
\end{eqnarray*}
Therefore, we are done by Lemma \ref{2-3} with $a= r-2\phi(c^{-2})^{-1}$.
\end{proof}

 \begin{lemma}
                                        \label{2-5}
Assume $2\phi(c^{-2})^{-1} < r$. Then for any $(t,x) \in Q_c(r)$,
\begin{eqnarray*}
&& \int_{r-\phi(c^{-2})^{-1}}^{r+\phi(c^{-2})^{-1}} \int_{\hat B_c}\int_{\hat B_c}  |\hat u(s,y_1)-\hat u(s,y_2)|^2 ~dy_1dy_2ds \\
 &\leq&
 N  \phi(c^{-2})^{-1} c^{2d}   [\bM_t \bM_x |f|_H^2 (t,x) +\cM_x \bM_t  \bM_x |f|_H^2 (t,x)].
\end{eqnarray*}
 where $N=N(d,T,\phi)$.
\end{lemma}

\begin{proof}
 By Minkowski's inequality,
\begin{eqnarray*}
&&|u(s,y_1) -u(s,y_2)|^2\\
&=&|(\int_0^s |\phi(\Delta)^{1/2} \hat T_{s-\kappa} f(\kappa,y_1) |_H^2 d\kappa)^{1/2}
-(\int_0^s |\phi(\Delta)^{1/2} \hat T_{s-\kappa} f(\kappa,y_2) |_H^2 d\kappa)^{1/2}|^2\\
&\leq& \int_0^s |\phi(\Delta)^{1/2} \hat T_{s-\kappa} f(\kappa,y_1)  - \phi(\Delta)^{1/2} \hat T_{s-\kappa} f(\kappa,y_2) |_H^2 d\kappa\\
&\leq& \int_0^{r-2\phi(c^{-2})^{-1}} |\phi(\Delta)^{1/2} \hat T_{s-\kappa} f(\kappa,y_1)  - \phi(\Delta)^{1/2} \hat T_{s-\kappa} f(\kappa,y_2) |_H^2 d\kappa\\
&&+\int_{r-2\phi(c^{-2})^{-1}}^s |\phi(\Delta)^{1/2} \hat T_{s-\kappa} f(\kappa,y_1)  - \phi(\Delta)^{1/2} \hat T_{s-\kappa} f(\kappa,y_2) |_H^2 d\kappa\\
&:=&I(s,y_1,y_2)+J(s,y_1,y_2).
\end{eqnarray*}

By Lemma \ref{2-3} with $a= r-2\phi(c^{-2})^{-1}$,
$$
\int_{r-\phi(c^{-2})^{-1}}^{r+\phi(c^{-2})^{-1}} \int_{\hat B_c}\int_{\hat B_c}  J(s)~dy_1dy_2ds \leq  N  \phi(c^{-2})^{-1} c^{2d}   [\bM_t \bM_x |f|_H^2 (t,x).
$$
Therefore, we only need to estimate $I$. Let $r-\phi(c^{-2})^{-1}<s< r+\phi(c^{-2})^{-1}$.
Observe that for $(s,y_1), (s,y_2) \in Q_c(r)$
\begin{equation}
                             \label{eqn 12.24.5}
I\leq Nc^{2d} \int_0^1 \int_{0}^{r-2\phi(c^{-2})^{-1}}  |\nabla \phi(\Delta)^{1/2} \hat T_{s-\kappa}  f(\kappa,\cdot)(\eta y_1+(1-\eta)y_2)|_H^2 d\kappa d\eta.
\end{equation}
Recall that $\hat p(s-\kappa,y)=0$ if $s-\kappa > T$. Therefore if $T+r-2\phi(c^{-2})^{-1} < s$, then
$$
c^{2d} \int_0^1 \int_{0}^{r-2\phi(c^{-2})^{-1}}  |\nabla \phi(\Delta)^{1/2} \hat T_{s-\kappa}  f(\kappa,\cdot)(\eta y_1+(1-\eta)y_2)|_H^2 d\kappa d\eta=0.
$$
%So we assume $T+r-2\phi(c^{-2})^{-1} \geq s$,
%which certainly implies
%\begin{eqnarray}
%                                   \label{eqn 12.24.1}
%T \geq s-(r-2\phi(c^{-2})^{-1}) \geq \phi(c^{-2})^{-1}, \quad c \leq \phi^{-1}(\frac{1}{T})^{-1/2}.
%\end{eqnarray}
So we assume $T+r-2\phi(c^{-2})^{-1} \geq s$,
which certainly implies
\begin{eqnarray*}
T \geq s-(r-2\phi(c^{-2})^{-1}) \geq \phi(c^{-2})^{-1}, \quad c \leq \phi^{-1}(\frac{1}{T})^{-1/2}.
\end{eqnarray*}
Moreover from \eqref{e:H1-large} and \eqref{e:H2-small} we see that
\begin{eqnarray}
                       \label{eqn 12.24.1}
c \leq \phi^{-1}(\frac{1}{T})^{-1/2} \leq \left ((\frac{1}{a_2 T})^{-1/(2\delta_2)} \vee (\frac{1}{a_3 T})^{-1/(2\delta_3)} \right).
\end{eqnarray}

Recall $a_t:=  \left(\phi^{-1}(t^{-1})\right)^{-1/2}$ and  $t\phi(a_t^{-2})=1$. For simplicity, denote
$$\bar{y}=\bar{y}(\eta)=\eta y_1+(1-\eta)y_2.
$$
  As before, using \eqref{4026}, we get
\begin{eqnarray*}
&&|\nabla \phi(\Delta)^{1/2} \hat T_{ s-\kappa}f(\kappa, \cdot)(\bar{y})|_H\\
&=&|\int_0^\infty \nabla \phi(\Delta)^{1/2} \hat p_{x^1} (s-\kappa ,\rho e_1)   \int_{|z| \leq \rho} f(\kappa ,\bar{y}-z)dz d\rho|_H\\
&\leq&|\int_0^\infty \nabla \phi(\Delta)^{1/2} \hat p_{x^1} (s-\kappa ,\rho e_1)   \int_{|\bar{y}-z| \leq \rho} f(\kappa ,z)dz d\rho|_H\\
&\leq& N \bM_x |f|_H(s,\bar{y}) \sum_{i=1}^d \int_0^\infty |\phi(\Delta)^{1/2} \hat p_{x^ix^1} (s-\kappa ,\rho e_1)|  \rho^d  d\rho\\
&\leq&N \bM_x |f|_H(s,\bar{y}) \sum_{i=1}^d [\int_0^{a_{s-\kappa} } | \phi(\Delta)^{1/2} \hat p_{x^ix^1} (s-\kappa ,\rho e_1)|  \rho^d d\rho \\
&&+\int_{a_{s-\kappa}}^\infty |\phi(\Delta)^{1/2} \hat p_{x^ix^1} (s-\kappa ,\rho e_1)| \rho^d d\rho]\\
&:=& I_1+I_2.
\end{eqnarray*}
By Lemma \ref{le1},
\begin{eqnarray}
I_1 &\leq& N \bM_x |f|_H(s,\bar y) ( s-\kappa)^{-1/2} \phi^{-1}((s-\kappa)^{-1})^{(d+2)/2} \int_0^{a_{s-\kappa} }   \rho^d d\rho \nonumber\\
&\leq& N \bM_x |f|_H(s,\bar y) ( s-\kappa)^{-1/2} \phi^{-1}((s-\kappa)^{-1})^{1/2}, \label{eqn 12.24.3}
\end{eqnarray}
and by Lemma \ref{le1} and Corollary \ref{cor116},
\begin{eqnarray}
I_2
&\leq&  N \bM_x |f|_H(s,\bar y)  \int_{a_{s-\kappa}}^\infty \frac{\phi(\rho^{-2})^{1/2} }{\rho^2}~d\rho \nonumber\\
&\leq&  N \bM_x |f|_H(s,\bar y) a_{s-\kappa}^{-1} \int_{a_{s-\kappa}}^\infty \frac{\phi(\rho^{-2})^{1/2} }{\rho}~d\rho \nonumber\\
&\leq&  N \bM_x |f|_H(s,\bar y) a_{s-\kappa}^{-1} \phi(a_{s-\kappa}^{-2})^{1/2}\nonumber\\
&=& N \bM_x |f|_H(s,\bar y) ( s-\kappa)^{-1/2} \phi^{-1}((s-\kappa)^{-1})^{1/2}. \label{eqn 12.24.4}
\end{eqnarray}
Therefore, using (\ref{eqn 12.24.3}) and (\ref{eqn 12.24.4}), and coming back to (\ref{eqn 12.24.5}), we get
\begin{eqnarray}
&& c^{2d} \int_0^{r-2\phi(c^{-2})^{-1}} \bM_x |f|^2_H(\kappa,\bar y) ( s-\kappa)^{-1} \phi^{-1}((s-\kappa)^{-1}) d\kappa \nonumber\\
&\leq&  N c^{2d}  \int_{-\infty}^{r-2\phi(c^{-2})^{-1}} \bM_x |f|^2_H(\kappa,\bar y)( r-\phi(c^{-2})^{-1}-\kappa)^{-1} \nonumber\\
&&\quad \quad \quad \quad \quad \quad \quad \quad \quad \quad \quad \quad \quad  \times \, \phi^{-1}((r-\phi(c^{-2})^{-1}-\kappa)^{-1}) \,\,d\kappa  \nonumber\\
&\leq&  N c^{2d}  \int_{\phi(c^{-2})^{-1}}^{\infty} \bM_x |f|^2_H(r-\phi(c^{-2})^{-1}-\kappa,\bar y) \kappa^{-1} \phi^{-1}(\kappa^{-1}) d\kappa  \nonumber\\
&\leq&  N c^{2d}  \int_{1}^{\infty} \bM_x |f|^2_H(r-\phi(c^{-2})^{-1}-\phi(c^{-2})^{-1}\kappa,\bar y) \kappa^{-1} \phi^{-1}( \phi(c^{-2})\kappa^{-1}) d\kappa  \nonumber\\
&\leq&  N c^{2(d-1)} \int_{1}^{\infty} \bM_x |f|^2_H(r-\phi(c^{-2})^{-1}-\phi(c^{-2})^{-1}\kappa,\bar y) \kappa^{-2}  d\kappa. \label{eqn 12.24.6}
\end{eqnarray}
 For the last inequality above, we used Lemma \ref{l:phi-property}. Indeed, for $\kappa \geq 1$ and $t>0$
$$
 \phi(t) \kappa^{-1} = \phi(\kappa t \kappa^{-1}) \kappa^{-1} \leq \phi(t \kappa^{-1}), \quad
\phi^{-1} (\phi(t) \kappa^{-1}) \leq t\kappa^{-1}.
$$
Note that  $|f|^2_H(r-\phi(c^{-2})^{-1}-\phi(c^{-2})^{-1}\kappa,\bar y)$  in (\ref{eqn 12.24.6}) can be replaced by $I_{\kappa>1}$ times of it.
Therefore, by integration by parts

\begin{eqnarray*}
&& c^{2(d-1)} \int_{1}^{\infty} \int_0^{\kappa} \bM_x |f|^2_H(r-\phi(c^{-2})^{-1}-\phi(c^{-2})^{-1}\nu,\bar y) d\nu ~\kappa^{-3}  d\kappa \\
&\leq&  N c^{2(d-1)}\phi(c^{-2})  \int_{1}^{\infty}
 \int_{r-\phi(c^{-2})^{-1}-\phi(c^{-2})^{-1}\kappa}^{r+\phi(c^{-2})^{-1}} \bM_x |f|^2_H(\nu,\bar y) d\nu ~\kappa^{-3}  d\kappa \\
&\leq&  N c^{2(d-1)} \bM_t\bM_x |f|^2_H(t,\bar y) \phi(c^{-2})  \int_{1}^{\infty} (2\phi(c^{-2})^{-1}+\phi(c^{-2})^{-1}\kappa) \kappa^{-3}  d\kappa \\
&\leq&  N c^{2(d-1)} \bM_t\bM_x |f|^2_H(t,\bar y)
\leq  N \bM_t\bM_x |f|^2_H(t,\bar y) ,
\end{eqnarray*}
where (\ref{eqn 12.24.1}) is used for the last inequality.
Thus,
$$
I(s,y_1,y_2) \leq N \int_0^1 \bM_t\bM_x |f|^2_H(t,\eta y_1+(1-\eta)y_2) d\eta,
$$
where $N$ depends only on $d, T, a_i,\delta_i$ $(i=1,2,3)$. Finally, we conclude that
\begin{eqnarray*}
     && \int_{r-\phi(c^{-2})^{-1}}^{r+\phi(c^{-2})^{-1}} \int_{\hat B_c}  \int_{\hat B_c}  I(s,y_1,y_2)\, dy_1dy_2ds \\
&\leq& N \phi(c^{-2})^{-1}\int_0^1 \int_{\hat B_c}  \int_{\hat B_c} \bM_t\bM_x |f|^2_H(t,\eta y_1+(1-\eta)y_2) dy_1 dy_2 d\eta \\
&\leq& N \phi(c^{-2})^{-1}\int_0^1 \int_{\hat B_c}  \int_{\eta \hat B_c+(1-\eta)y_2} \bM_t\bM_x |f|^2_H(t,y)dy dy_2 d\eta \\
&\leq& N \phi(c^{-2})^{-1} \int_{\hat B_c}  \int_{\hat B_c} \bM_t\bM_x |f|^2_H(t,y) dy dy_2 \\
&\leq& N \phi(c^{-2})^{-1} c^{2d} \cM_x\bM_t \bM_x |f|_H^2 (t,x).
\end{eqnarray*}
The lemma is proved.
\end{proof}

% Let $\cF$ be the collection of all $Q_c(r,z)$ where $c>0$ and $(r,z) \in \bR^{d+1}$.
For a measurable function $h(t,x)$  on $\bR^{d+1}$ define the sharp function
$$
h^{\#}(t,x) := \sup_{Q_c(r,z) \ni (t,x)}  \frac{1}{| Q_c(r,z)|}  \int_{Q_c(r,z)} |h (s,y)-h_{Q_c(r,z)}|~dsdy,
$$
where
$$
h_{Q_c(r,z)}(x) = \dashint_{Q_c(r,z)} h(s,y)~dsdy:=  \frac{1}{| Q_c(r,z)|} \int_{Q_c(r,z)} h(s,y)~dsdy.$$

The following two theorems are classical results and can be found in \cite{Ste}.

\begin{thm} (Hardy-Littlewood)
\label{hl}
For $1 < p <\infty$ and $f \in L_p(\bR^d)$,
we have
%$$
%\|\cM_x f \|_{L_p(\bR^d)}+\|\bM_x f \|_{L_p(\bR^d)} \leq N(d,p) \|f\|_{L_p(\bR^d)}.
%$$
$$
\|\cM_x f \|_{L_p(\bR^d)}+\|\bM_x f \|_{L_p(\bR^d)} \leq N(d,p,\phi) \|f\|_{L_p(\bR^d)}.
$$
\end{thm}

\begin{thm}(Fefferman-Stein). \label{fst}
For any $1<p<\infty$ and $h\in L_p(\bR^{d+1})$,
%\begin{equation*}
%\|h\|_{L_p(\bR^{d+1})}\leq N(d,p)\|h^\#\|_{L_p(\bR^{d+1})}.
%\end{equation*}
\begin{equation*}
\|h\|_{L_p(\bR^{d+1})}\leq N(d,p,\phi)\|h^\#\|_{L_p(\bR^{d+1})}.
\end{equation*}
\end{thm}

\begin{proof}
 We can get this result from Theorem IV.2.2
in \cite{Ste}. Indeed, due to \eqref{e:sc1} we can easily check that  the balls $Q_c(s,y)$ satisfy the
conditions (i)-(iv) in section 1.1 of \cite{Ste} :\\
(i) $Q_{c}(t,x)\cap Q_{c}(s,y)\neq \emptyset$ implies
$Q_{c}(s,y)\subset Q_{N_1c}(t,x)$ ;\\
(ii) $|Q_{N_1c}(t,x)|\leq N_2 |Q_{c}(t,x)|$ ;\\
(iii) $\cap_{c>0}\overline{Q}_{c}(t,x)=\{(t,x)\}$ and
$\cup_{c}Q_{c}(t,x)=\bR^{d+1}$ ;\\
(iv) for each open set $U$ and $c>0$, the function $(t,x)\to
|Q_{c}(t,x)\cap U|$ is continuous.
\end{proof}

\vspace{4mm}

{\bf{Proof of Theorem \ref{main theorem}}}.

First assume $f(t,x)=0$ if $t\not\in [0,T]$.

Since the theorem is already proved if $p=2$ in Lemma \ref{2-1}, we assume $p>2$.

First, we prove
\begin{equation}
                           \label{eqn 6.08.9}
((\hat \cG f )^{\#})^2 (t,x) \leq N (G(t,x)+G(-t,x)),
 \end{equation}
where
$$
G(t,x):=\bM_t\bM_x| f|^2_{H}(t,x)+\cM_x\bM_t\bM_x|f|^2_{H} (t,x)+ \cM_x \bM_t  \bM_x |f|_H^2 (t-T,x).
$$
Because of Jensen's inequality, to prove (\ref{eqn 6.08.9}) it suffices to
prove that for each $Q_c(r,z) \in \cF$ and $(t,x)\in Q_c(r,z)$,
\begin{eqnarray}
 && \dashint_{Q_c(r,z)} \dashint_{Q_c(r,z)} |\hat \cG f(s_1,y_1)-\hat \cG f (s_2,y_2)|^2 ~ds_1dy_1ds_2dy_2 \nonumber \\
 &\leq&  N     (G(t,x)+G(-t,x))  \label{4033}.
\end{eqnarray}
To prove this we use translation and apply Lemma \ref{2-3}, \ref{2-4}, and \ref{2-5}.

By the definition of $\hat \cG f$ and the fact  $\hat T_t g (y+z) = \hat T_t g(z+ \cdot)(y)$, we see for $s \geq 0$
\begin{eqnarray}
\hat\cG f(s,y+z)
 &=& [\int_{0}^{s} |\phi(\Delta)^{1/2}\hat T_{s-\rho}f(\rho,\cdot)(y+z)|_H^2 d \rho]^{1/2} \nonumber\\
 &=& [\int_{0}^{s} |\phi(\Delta)^{1/2}\hat T_{s-\rho}f(\rho,z+\cdot)(y)|_H^2 d \rho]^{1/2} \nonumber\\
&=& [\int_{0}^{s} |\phi(\Delta)^{1/2}\hat T_{s-\rho}f(\rho,z+\cdot)(y)|_H^2 d \rho]^{1/2} \nonumber\\
&=&\hat \cG f( \cdot, z + \cdot) (s,y), \label{12051}
\end{eqnarray}
and
$$
\hat \cG f(-s,y+z)=\hat \cG f(s,y+z)=\hat \cG f( \cdot, z + \cdot) (s,y)=\hat \cG f( \cdot, z + \cdot) (-s,y).
$$
Therefore we get
\begin{eqnarray*}
 \dashint_{Q_c(r,z)} |\hat \cG \{f( \cdot,\cdot)\} (s,y) |^2~dyds
&=& \frac{1}{|Q_c(r)|}\int_{Q_c(r)} |\hat \cG \{f ( \cdot,\cdot)\} (s,y+z) |^2~dyds\\
&=& \frac{1}{|Q_c(r)|}\int_{Q_c(r)} |\hat \cG \{f ( \cdot,z+\cdot)\} (s,y) |^2~dyds.
\end{eqnarray*}
This shows that we may assume  that $z=0$ and $Q_c(r,z)=Q_c(r)$.

If  $|r| \leq 2\phi(c^{-2})^{-1}$, then (\ref{4033}) follows from Lemma \ref{2-3}.  Also if $r>2\phi(c^{-2})^{-1}$ then (\ref{4033}) follows from Lemmas
\ref{2-4} and \ref{2-5}. Therefore it only remains to consider the case  $r<-2\phi(c^{-2})^{-1}$.  In this case, (\ref{4033}) follows from the identity
\begin{eqnarray*}
&&\dashint_{Q_c(r)} \dashint_{Q_c(r)} |\hat \cG f(s_1,y_1)-\hat \cG f (s_2,y_2)|^2 ~ds_1dy_1ds_2dy_2\\
&=&\dashint_{Q_c(-r)} \dashint_{Q_c(-r)} |\hat \cG f(s_1,y_1)-\hat \cG f (s_2,y_2)|^2 ~ds_1dy_1ds_2dy_2.
\end{eqnarray*}
This is because, for $r':=-r$, we have $r'>2\phi(c^{-2})^{-1}$ and this case is already proved above. Thus we have proved (\ref{eqn 6.08.9}).

Recall that $\cG f (t,x)=\hat{\cG}f(t,x)$ for $t\in [0,T]$.  Thus by Theorem \ref{fst} and (\ref{eqn 6.08.9})
%\begin{eqnarray*}
%&&\|\cG f\|_{L_p([0,T] \times \bR^{d})}^p
%=\|I_{[0,T]} \hat \cG  f  \|_{ L_p( \bR^{d+1})}^p
%\leq N\|(\hat \cG f)^{\#} \|_{L_p(\bR^{d+1})}^p\\
%&\leq&  N \|(\bM_t\bM_x|f|^2_{H})^{1/2}(t,x)+(\cM_x \bM_t\bM_x|f |^2_{H})^{1/2}(t,x)\|_{L_p( \bR^{d+1})}^p,
%\end{eqnarray*}
\begin{eqnarray*}
&&\|\cG f\|_{L_p([0,T] \times \bR^{d})}^p
=\|I_{[0,T]} \hat \cG  f  \|_{ L_p( \bR^{d+1})}^p
\leq N\|(\hat \cG f)^{\#} \|_{L_p(\bR^{d+1})}^p\\
&\leq&  N \left(\|(\bM_t\bM_x|f|^2_{H})^{1/2}(t,x)\|_{L_p( \bR^{d+1})}^p+\|(\cM_x \bM_t\bM_x|f |^2_{H})^{1/2}(t,x)\|_{L_p( \bR^{d+1})}^p \right),
\end{eqnarray*}
where for the last inequality we use (\ref{eqn 6.08.9}) and the fact that the $L_p$-norm is invariant under reflection and translation.
Now we use
 Theorem \ref{hl}  to get
\begin{eqnarray*}
&& \int_{\bR^d} \int_{-\infty}^\infty (\bM_t \bM_x | f |_H^2)^{p/2}~dtdx  +\int_{-\infty}^\infty\int_{\bR^d}  (\cM_x \bM_t \bM_x |f |_H^2)^{p/2}~dtdx\\
&\leq&N \int_{-\infty}^\infty \int_{\bR^d} ( \bM_x | f|_H^2)^{p/2}~dtdx  + N\int_{\bR^d} \int_{-\infty}^\infty ( \bM_t \bM_x | f|_H^2)^{p/2}~dtdx\\
&\leq&N \int_{-\infty}^\infty \int_{\bR^d} (  | f|_H^2)^{p/2}~dtdx  +N \int_{-\infty}^\infty\int_{\bR^d}  (  \bM_x | f|_H^2)^{p/2}~dtdx\\
&\leq&  N \|f\|_{L_p( \bR^{d+1},H)}^p=N \|f\|_{L_p([0,T]\times \bR^{d},H)}^p.
\end{eqnarray*}

Finally, for the general case,  choose $\zeta_n \in C_0^{\infty}(0,T)$ such that $|\zeta_n| \leq 1$, $\zeta_n= 1$ on $[1/n,T-1/n]$. Then from the above inequality and Lebesgue dominated convergence theorem we conclude that
\begin{eqnarray*}
&&\|\cG f\|_{L_p([0,T] \times \bR^{d})}
\leq \limsup_{n \to \infty} \|\cG (f\zeta_n)\|_{L_p([0,T] \times \bR^{d})}\\
&&\leq \limsup_{n \to \infty} \|f \zeta_n\|_{L_p(\bR^{d+1},H)}
\leq \|f\|_{L_p([0,T] \times \bR^d,H)}.
\end{eqnarray*}
Hence the theorem is proved.
\qed

\mysection{An application}
                            \label{section application}

In this section we construct an $L_p$-theory for the stochastic  integro-differential equation
\begin{equation}
                    \label{eqn 0-1}
      du=(\phi(\Delta)u+f)dt+ g^k dw^k_t, \quad u(0)=0.
 \end{equation}
The condition $u(0)=0$ is not necessary and is assumed only for the simplicity of the presentation.

 %Let $(\Omega,\cF,P)$ be a complete probability space, and
%$\{\cF_{t},t\geq0\}$ be an increasing filtration of
%$\sigma$-fields $\cF_{t}\subset\cF$, each of which contains all
%$(\cF,P)$-null sets. By  $\cP$ we denote the predictable
%$\sigma$-field generated by $\{\cF_{t},t\geq0\}$ and  we assume that
%on $\Omega$ we are given independent one-dimensional Wiener
%processes $w^{1}_{t},w^{2}_{t},...$, each of which is a Wiener
%process relative to $\{\cF_{t},t\geq0\}$.
%
%
%
%
%For $\gamma\in \bR$,
%denote $H^{\phi,\gamma}_{p}=(1-\phi(\Delta))^{-\gamma/2}L_p$, that is $u\in H^{\phi, \gamma}_p$ if
%\begin{equation}
%                   \label{norm}
%\|u\|_{H^{\phi,\gamma}_p}=\|(1-\phi(\Delta))^{\gamma/2}u\|_p:=\|\cF^{-1}\{(1+\phi(|\xi|^2))^{\gamma/2} \cF(u)(\xi)\}\|_p<\infty,
%\end{equation}
%where $\cF$ is the Fourier transform.

Let $(\Omega,\cI,P)$ be a complete probability space, and
$\{\cI_{t},t\geq0\}$ be an increasing filtration of
$\sigma$-fields $\cI_{t}\subset\cI$, each of which contains all
$(\cI,P)$-null sets. By  $\cP$ we denote the predictable
$\sigma$-field generated by $\{\cI_{t},t\geq0\}$ and  we assume that
on $\Omega$ we are given independent one-dimensional Wiener
processes $w^{1}_{t},w^{2}_{t},...$, each of which is a Wiener
process relative to $\{\cI_{t},t\geq0\}$.

For $\gamma\in \bR$,
denote $H^{\phi,\gamma}_{p}=(1-\phi(\Delta))^{-\gamma/2}L_p$, that is $u\in H^{\phi, \gamma}_p$ if
\begin{equation}
                   \label{norm}
\|u\|_{H^{\phi,\gamma}_p}=\|(1-\phi(\Delta))^{\gamma/2}u\|_p:=\|\cF^{-1}\{(1+\phi(|\xi|^2))^{\gamma/2} \cF(u)(\xi)\}\|_p<\infty,
\end{equation}
where $\cF$ is the Fourier transform and $\cF^{-1}$ is the inverse Fourier transform. Similarly, for a $\ell_2$-valued function $u=(u^1,u^2,\cdots)$ we define
$$
\|u\|_{H^{\phi,\gamma}_p(\ell_2)}=\||\cF^{-1}\{(1+\phi(|\xi|^2))^{\gamma/2} \cF(u)(\xi)\}|_{\ell_2}\|_p.
$$

The following result can be found, for instance, in \cite{FJS}.

\begin{lemma}
                         \label{lemma banach}
  (i) For any $\mu, \gamma \in \bR$, the map $(1-\phi(\Delta))^{\mu/2}: H^{\phi,\gamma}_p \to H^{\phi,\gamma-\mu}_p$ is  an isometry.

  (ii) For any $\gamma\in \bR$, $H^{\phi,\gamma}_p$ is a Banach space.

  (iii) If $\gamma_1\leq \gamma_2$, then $H^{\phi, \gamma_2}_p \subset H^{\phi, \gamma_1}_p$ and
  $$
  \|u\|_{H^{\phi,\gamma_1}_p}\leq c\|u\|_{H^{\phi,\gamma_2}_p}.
  $$

  (iv)   Let $\gamma \geq 0$. Then there is a constant $c>1$ so that
  $$
 c^{-1} \|u\|_{H^{\phi,\gamma}_p}    \leq (\|u\|_p+ \|\phi(\Delta)^{\gamma/2}u\|_p)\leq c \|u\|_{H^{\phi,\gamma}_p}.
 $$
\end{lemma}
\begin{proof}
(i) follows from definition (\ref{norm}). For (ii), it suffices to prove the completeness. Let $\{u_n:n=1,2,\cdots\}$ be a Cauchy sequence in $H^{\phi,\gamma}_p$.
Then $f_n:=(1-\phi(\Delta))^{\gamma/2}u_n$ is a Cauchy sequence in $L_p$, and there exists $f\in L_p$ so that $f_n\to f$ in $L_p$. Define $u:=(1-\phi(\Delta))^{-\gamma/2}f$. Then $u\in H^{\phi,\gamma}_p$ and
$$
\|u_n-u\|_{H^{\phi,\gamma}_p}=\|f_n-f\|_p \to 0.
$$
Finally (iii) and (iv) are consequences of a Fourier multiplier theorem. Indeed, due to Theorem 0.2.6 in \cite{sogge}, we only need to show that for any $\gamma_1 \leq 0$ and $\gamma_2 \in \bR$
%$$
%|D^n(1+\phi(|\xi|^2)^{\gamma_1}| + |D^n[\frac{(1+\phi(|\xi|^2))^{\gamma_2}}{1+\phi(|\xi|^2)^{\gamma_2}}]| \leq N(n) |\xi|^{-n}.
%$$
$$
|D^n[1+\phi(|\xi|^2)^{\gamma_1}]| + |D^n[\frac{(1+\phi(|\xi|^2))^{\gamma_2}}{1+\phi(|\xi|^2)^{\gamma_2}}]|+|D^n[\frac{(\phi(|\xi|^2))^{\gamma_2}}{(1+\phi(|\xi|^2))^{\gamma_2}}]| \leq N(n) |\xi|^{-n}.
$$

This comes from Lemma \ref{13131}. The lemma is proved.
\end{proof}

Denote
$$
\bH^{\phi,\gamma}_p(T)=L_p(\Omega\times [0,T], \cP, H^{\phi,\gamma}_p), \quad \bL_p(T):=\bH^{\phi,0}_p(T),
$$
$$
\bH^{\phi,\gamma}_p(T,\ell_2)=L_p(\Omega\times [0,T], \cP, H^{\phi,\gamma}_p(\ell_2)).
$$

\begin{defn}
 We write   $u\in \cH^{\phi,\gamma+2}_{p,\theta}(T)$  if
 $u\in \bH^{\phi,\gamma+2}_{p,\theta}(T)$, $u(0)=0$ and  for some $f \in\bH^{\phi,\gamma}_{p,\theta}(T)$ and
 $g=(g^1,g^2,\cdots)\in \bH^{\phi,\gamma+1}_{p}(T,\ell_2)$,  it holds that
 $$
 du=fdt+g^kdw^k_t
 $$
in the sense of distributions, that is for any $\varphi \in C^{\infty}_0(\bR^d)$, the equality
\begin{equation}
                              \label{distribution}
(u(t),\varphi)=\int^t_0 (f(s),\varphi)ds+\sum_k \int^t_0 (g^k(s),\varphi)dw^k_s
\end{equation}
holds for all $t\leq T$ (a.s.). In this case we write
$$
f=\bD u, \quad g=\bS u.
$$
The norm in  $
\cH^{\phi,\gamma+2}_{p}(T)$ is given by
$$
\|u\|_{\cH^{\phi,\gamma+2}_{p}(T)}=
\|u\|_{\bH^{\phi,\gamma+2}_{p}(T)} + \|\bD u\|_{\bH^{\phi,\gamma}_{p}(T)}  + \|\bS u\|_{\bH^{\phi,\gamma+1}_{p}(T,\ell_2)}.
$$
\end{defn}

\begin{remark}
As explained in \cite[Remark 3.2]{Kr99}, for any $g\in \bH^{\phi,\gamma+1}_p(T,\ell_2)$ and $\varphi\in C^{\infty}_0(\bR^d)$, the series of stochastic integral $\sum_k (g^k,\varphi)dw^k_t$ converges in probability uniformly in $[0,T]$ and  defines a continuous square integrable martingale on $[0,T]$.
\end{remark}

\begin{thm}
                             \label{thm banach}
For any $\gamma$ and $p\geq 2$,  $\cH^{\phi,\gamma+2}_p(T)$ is a Banach space. Also, for any $u\in \cH^{\phi,\gamma+2}_p(T)$,  we have $u\in C([0,T],H^{\phi,\gamma}_p)$ (a.s.) and
 \begin{equation}
                         \label{eqn 1.1}
 \E \sup_{t\leq T}\|u(t,\cdot)\|^p_{H^{\phi,\gamma}_p}\leq N\left(\|\bD u\|^p_{\bH^{\phi,\gamma}_p(T)}+\|\bS u\|^p_{\bH^{\phi,\gamma}_p(T,\ell_2)}\right).
\mathrm{}\end{equation}
In particular, for $t\leq T$,
$$
\|u\|^p_{\bH^{\phi,\gamma}_p(t)}\leq N \int^t_0 \|u\|^p_{\cH^{\phi,\gamma+2}_p(s)}\,ds.
$$
\end{thm}
\begin{proof}
Since  the operator $(1-\phi(\Delta))^{\gamma/2}: \cH^{\phi,\gamma+2}_p(T)\to \cH^{\phi,2}_p(T)$ is an isometry, it suffices to prove the case $\gamma=0$. In this case $H^{\phi,\gamma}_p=L_p$, and therefore
(\ref{eqn 1.1}) is proved, for instance, in Theorem 3.4 of \cite{KK}. Also, the completeness of the space  $\cH^{\phi,\gamma+2}_{p}(T)$ can be proved using (\ref{eqn 1.1}) as in the proof of Theorem 3.4 of \cite{KK}.
\end{proof}

 The following maximal principle  will be used  to prove the uniqueness result of equation (\ref{eqn 0-1}).

\begin{lemma}
                          \label{maxip}
Let $\lambda >0$ be a constant. Suppose that $u$ is continuous in $[0,T]\times \bR^d$, $u(t,\cdot)\in C^2_b(\bR^d)$ for each $t>0$, $u_t, \phi(\Delta)u$ are continuous in $(0,T]\times \bR^d)$, $u_t - \phi(\Delta) u +\lambda u=0$ for $t\in (0,T]$, $u(t_n,x) \to u(t,x)$ as $t_n \to t$ uniformly for $x \in \bR^d$, $u(0,x)=0$ for all $ x \in \bR^d$, and for each $t$ $u(t,x) \to 0$ as $|x| \to \infty$. Then $u \equiv 0$ in $[0,T] \times \bR^d$.
\end{lemma}

\begin{proof}
Suppose $\sup_{(t,x) \in [0,T]\times \bR^d} u(t,x) >0$. Then we claim that there exists a $(t_0,x_0) \in [0,T] \times \bR^d$ such that
$ u (t_0,x_0) = \sup_{(t,x) \in [0,T]\times \bR^d} u(t,x)$. The explanation is as follows. Since there exists a sequence $(t_n,x_n)$ such that
 $u(t_n,x_n) \to \sup_{(t,x) \in [0,T]\times \bR^d} u(t,x)$, one can choose a subsequence $t_{n_k}$ such that $t_{n_k} \to t_0$ for some $t_0 \in [0,T]$ as $k \to \infty$. If $\{x_n\}$ is unbounded, then there exists a subsequence $x_{n_k}$ such that $|x_{n_k}| \to \infty$. Due to the assumption :
  $u(t_n,x) \to u(t,x)$ as $t_n \to t$ uniformly for $x \in \bR^d$,
 we have $u(t_{n_k},x_{n_k}) \to u(t_0, x_{n_k})$ as $k \to \infty$. But since $u(t_0, x_{n_k}) \to 0$ as $k \to \infty$ this is contradiction to the fact
 $u(t_{n_k}, x_{n_k}) \to \sup_{(t,x) \in [0,T]\times \bR^d} u(t,x) >0$. Therefore, $\{x_n\}$ is bounded and this means that $x_n$ also has a subsequence $x_{n_k}$ such that $x_{n_k} \to x_0$ for some $x_0 \in \bR^d$. So we know that
 our claim is true. Note that since $u(0,x)=0 ~ \forall x \in \bR^d$, $ t_0>0$ and if $t_0 \in (0,T)$ then $ u_t(t_0,x_0)=0$ otherwise $u_t(T,x_0) \geq 0$. Recall that
\begin{align*}
&\phi(\Delta) u(t,x) = \lim_{\eps \downarrow 0}\int_{\{y\in \bR^d: \, |y|>\eps\}} (u(t,x+y)-u(t,x)) j(|y|)\, dy
\end{align*}
and $j$ is strictly positive. Therefore, we get $(u_t -\phi(\Delta) u +\lambda u)(t_0,x_0) > 0$ and this is contradiction to our assumption. So we have $\sup_{(t,x) \in [0,T]\times \bR^d} u(t,x) \leq 0$. Similarly, we can easily show that $\inf_{(t,x) \in [0,T]\times \bR^d} u(t,x) \geq 0$. The lemma is proved.
\end{proof}

The following result will be used to estimate the deterministic part of (\ref{eqn 0-1}).

\begin{lemma}
\label{multi}
Let $m(\tau, \xi):=\frac{ \phi(|\xi|^2)}{i\tau + \phi(|\xi|^2)}$. Then, $m$ is a $L_p(\bR^{d+1})$-multiplier. In other words,
$$
\|\cF^{-1}( m  \cF f)\|_{L_p(\bR^{d+1})} \leq N \|f\|_{L_p(\bR^{d+1})}, \quad \forall f\in L_p(\bR^{d+1}),
$$
where $N$ depends only on $d$ and $p$.
\end{lemma}
\begin{proof}
First we estimate derivatives of $m$.
Let $\alpha=(\alpha_1,\cdots, \alpha_{d})\neq 0$ be a $d$-dimensional multi-index with  $\alpha_i = 0$ or $1$ for $i =1, \ldots,d$. Assume $\beta,\gamma$ are multi-indices so that $\beta+\gamma=\alpha$.
%
%Let $\beta=(\beta_1,\cdots,\beta_{d})=(\beta_1,\beta')$ and $\gamma=(\gamma_1,\cdots,\gamma_{d})=(\gamma_1,\gamma')$  so that $(0,\beta) + (0, \gamma)=(0,\beta_1,\beta')+(0,\gamma_1,\gamma') = \alpha$.
Then from Lemma \ref{13131} we can easily get
\begin{eqnarray}
\label{2012}|D^{\beta} (\phi(|\xi|^2)) | \leq N \phi(|\xi|^2) |\xi|^{-|\beta|}.
\end{eqnarray}
Suppose $\gamma \neq 0$. Without loss of generality  assume $\gamma_1 = 1$.
Then by Leibniz's rule and \eqref{2012}, we get
\begin{eqnarray}
\notag &&|D^\gamma (i\tau +\phi(|\xi|^2))^{-1} |\\
\notag &=&|D^{\gamma'} D^{\gamma_1} (i\tau +\phi(|\xi|^2))^{-1} |\\
\notag &=&|D^{\gamma'}(i\tau +\phi(|\xi|^2))^{-2} D_{\xi^1}(\phi(|\xi|^2)) |\\
\notag &\leq&N \sum_{ \bar \gamma'+ \hat \gamma'=\gamma'} |(i\tau +\phi(|\xi|^2))^{-(|\bar \gamma'|+2)} \\
\notag &&\quad \quad \quad \quad \quad
\times [D_{\xi^1}(\phi(|\xi|^2))]^{ \bar \gamma_1'}\cdots [D_{\xi^d}(\phi(|\xi|^2))]^{ \bar \gamma_d'}D^{ \hat \gamma'} D_{\xi^1}(\phi(|\xi|^2)) |\\
\notag &\leq&N\sum_{\bar \gamma' + \hat \gamma' =\gamma'} |\tau^2+ \phi(|\xi|^2)^2|^{-(|\bar \gamma' |+2)/2} \phi(|\xi|^2)^{|\bar \gamma'|}|\xi|^{-|\bar \gamma'|} \phi(|\xi|^2) |\xi|^{-(|\hat \gamma'|+1)}\\
\label{201}&\leq&N\   |\tau^2+ \phi(|\xi|^2)^2|^{-1/2}   |\xi|^{-|\gamma|}.
\end{eqnarray}
Obviously even if $\gamma=0$, \eqref{201} holds. Therefore from \eqref{2012} and \eqref{201}, we get
\begin{eqnarray}
\label{202}|D^\alpha m(\tau,\xi) | \leq N \frac{ \phi(|\xi|^2)}{ |\tau| + \phi(|\xi|^2)} |\xi|^{-|\alpha|}.
\end{eqnarray}

Next let $\hat{\alpha} = (\alpha_0, \alpha_1, \ldots, \alpha_{d})=(\alpha_0, \alpha)$ be a $(d+1)$-dimensional multi-index with $\alpha_0 \neq 0$
 and $\alpha_i=0$ or $1$ for $i=1, \ldots,d$.
 Then from \eqref{202} we get
\begin{eqnarray}
\notag |D^{\hat{\alpha}} m(\tau,\xi)|
\notag &\leq& N(|\tau| +  \phi(|\xi|^2))^{-\alpha_0} |D^{\alpha} m(\tau,\xi)|\\
\label{203} &\leq& N \frac{ \phi(|\xi|^2)}{ (|\tau| + \phi(|\xi|^2))^{\alpha_0 +1}} |\xi|^{-|\alpha|}.
\end{eqnarray}
Now to conclude that $m$ is a multiplier, we use Theorem 4.6' in p 109 of \cite{Ste2}. Due to \eqref{203}, we see that for each $0 < k \leq d+1$
\begin{eqnarray*}
|\frac{\partial^k m}{\partial \tau  \partial \xi_1  \cdots \partial \xi_{k-1}}|
\leq N \frac{ \phi(|\xi|^2)}{ (|\tau| + \phi(|\xi|^2))^{2}} |\xi|^{-|k-1|}.
\end{eqnarray*}
Therefore, for any dyadic rectangles $A= \Pi_{1 \leq i \leq k} [2^{k_i}, 2^{k_i+1}]$, we have
\begin{eqnarray*}
\int_{\Pi} |\frac{\partial^k m}{\partial x_1 \partial x_2  \cdots \partial x_k}|  d\tau d\xi_1 \cdots \xi_{k-1}
\leq N.
\end{eqnarray*}
We can easily check that from \eqref{202} and \eqref{203}, the above statement is also valid for every one of the $n!$ permutations of the variables $\tau,\xi_1,\ldots,\xi_d$. The lemma is proved.
\end{proof}
Here is our $L_p$-theory.

\begin{thm}
            \label{Lptheory}
 For any $f \in\bH^{\phi,\gamma}_{p}(T)$ and
 $g=(g^1,g^2,\cdots)\in \bH^{\phi,\gamma+1}_{p}(T,\ell_2)$,  equation (\ref{eqn 0-1})
      has a unique solution $u\in \cH^{\phi,\gamma+2}_p(T)$, and for this solution
\begin{equation}
                            \label{eqn 1.3}
\|u\|_{\cH^{\phi,\gamma+2}_p(T)}\leq N\|f\|_{\bH^{\phi,\gamma}_p(T)}+N\|g\|_{\bH^{\phi,\gamma+1}_p(T,\ell_2)}.
\end{equation}
 \end{thm}

\begin{proof}
Due to the isometry, we may assume $\gamma=0$.  Note that if $u$ is a solution of \eqref{eqn 0-1} in $\cH^{\phi, 2}_p(T)$, then we have $u\in C([0,T], L_p)$ (a.s.) by Theorem \ref{thm banach}.

{\bf{Step 1}}. First we prove the uniqueness of Equation \eqref{eqn 0-1}. Let $u_1$, $u_2$ be solutions of equation \eqref{eqn 0-1}. Then putting $v:= u_1- u_2$, we see that
$v$ satisfies \eqref{eqn 0-1} with $f=g^k=0$.

Take a non-negative smooth function $\varphi\in C^{\infty}_0$ with unit integral. For $\varepsilon>0$, define $\varphi_{\varepsilon}(x)=\varepsilon^{-d}\varphi(x/\varepsilon)$. Also denote $w(t,x):= e^{-\lambda t}u(t,x)$ and $w^{\varepsilon}=w * \varphi_{\varepsilon}$.
Then by plugging $\varphi_{\varepsilon}(\cdot-x)$ in (\ref{distribution}) in place of $\varphi$, we have $w^{\varepsilon}_t-\phi(\Delta)w^{\varepsilon}+\lambda w^{\varepsilon}=0$. Also one can easily check that $w^{\varepsilon}$ satisfies the conditions in Lemma \ref{maxip} and concludes that $w^{\varepsilon}\equiv 0$. This certainly proves $v\equiv 0$ (a.s.).

{\bf{Step 2}}. We consider the case $g=0$. By  approximation argument (see the next step for the detail), to prove the existence and (\ref{eqn 1.3}), we may assume that
$f$ is sufficiently smooth in $x$ and vanishes if $|x|$ is sufficiently large. In this case, one can easily check that for each  $\omega$,
\begin{equation}
                               \label{eqn 1.3.5}
u(t,x)=\int^t_0 T_{t-s} f(s,x)ds
\end{equation}
satisfies $u_t=\phi(\Delta)u+f$. In addition to it, denoting $\bar p (t,x) = I_{0 \leq t } p(t,x)$ and $\bar f = I_{0 \leq t \leq T} f(t,x)$, we see that for $(t,x) \in [0,T] \times \bR^d$
 \begin{eqnarray}
\label{01312}
u(t,x)= \bar p( \cdot , \cdot ) \ast \bar f(\cdot, \cdot) (t,x).
\end{eqnarray}

 We use notation $\cF_d$ and $\cF_{d+1}$ to denote the Fourier transform for $x$ and $(t,x)$, respectively.  Moreover for convenience we put $\cF_d u(t,x) = \cF_d (u(t, \cdot))(x)$. Under this setting, observe that
\begin{eqnarray*}
\cF_{d+1} (\bar p) (\tau, \xi)
 = \int_{\bR} e^{-i\tau t} \cF_d( \bar p) (t,\xi) dt
=\int_{0}^\infty e^{-i\tau t} e^{-t \phi(|\xi|^2)} dt
=\frac{1}{i\tau + \phi(|\xi|^2)}.
\end{eqnarray*}
So denoting $\bar u = \bar p( \cdot , \cdot ) \ast \bar f(\cdot, \cdot) (t,x)$ we see
\begin{eqnarray}
\notag \cF_{d+1}^{-1} [(1+ \phi(|\xi|^2)) \cF_{d+1} \bar u](t,x)
\notag &=& \bar u +\cF_{d+1}^{-1} [ \phi(|\xi|^2) \cF_{d+1} (\bar p )  \cF_{d+1} (\bar f )]\\
\label{01315}&=& \bar u+ \cF_{d+1}^{-1} [\frac{ \phi(|\xi|^2)}{i\tau + \phi(|\xi|^2)}   \cF_{d+1} (\bar f )].
\end{eqnarray}
Due to generalized Minkowski's inequality, we can easily check that $\| \bar u\|_{L_p(\bR^{d+1})}\leq N\| f\|_{L_p([0,T]\times \bR^d)}$. Moreover we know that
$\frac{\phi(|\xi|^2)}{i\tau + \phi(|\xi|^2)}$ is a $L_p(\bR^{d+1})$-multiplier from Lemma \ref{multi}. Therefore, from \eqref{01312} we conclude that
\begin{eqnarray*}
\|u\|_{\bH^{\phi,2}_p(T)} \leq \| f \|_{\bL_p(T)}.
\end{eqnarray*}

{\bf{Step 3}}. We consider the case $f=0$.
First, assume that $g^k=0$ for all sufficiently large $k$ (say for all $k\geq N_0$), and each $g^k$ is of the type
\begin{equation}
                       \label{eqn 03.17.11}
g^k(t,x)=\sum_{i=0}^{m_k}{\bf 1}_{(\tau^k_i,\tau^k_{i+1}]}(t)g^{k_i}(x) \qquad \text{for } k \le N_0,
\end{equation}
where $\tau^k_i$ are bounded stopping times and $g^{k_i}(x)\in
 C^{\infty}_0 (\bR^d)$. Define
$$
v(t,x):=\sum_{k=1}^{N_0} \int^t_0
 g^k(s, x) dw^k_s
=\sum_{k=1}^{N_0}\sum_{i=1}^{m_k} g^{k_i}(x)(w^k_{t\wedge
\tau^k_{i+1}}-w^k_{t\wedge \tau^k_{i}})
$$
and
\begin{equation}
                    \label{eqn 03-22-1}
u(t,x):=v(t,x)+\int^t_0 \phi(\Delta)T_{t-s} v (s,x)\,
ds=v(t,x)+\int^t_0 T_{t-s}\phi(\Delta)v (s,x)\, ds.
\end{equation}
Then $u-v=\int^t_0 T_{t-s}\phi(\Delta)v (s,x)ds$, and
therefore (see (\ref{eqn 1.3.5})) we have
$$(u-v)_t=\phi(\Delta)(u-v)+ \phi(\Delta)v=\phi(\Delta)u,
$$
and
$$
du=d(u-v) +dv=\phi(\Delta) u dt + \sum_{k=1}^{N_0} g^kdw^k_t.
$$
Also by (\ref{eqn 03-22-1}) and stochastic Fubini theorem (\cite[Theorem 64]{P}), almost
surely,
\begin{eqnarray*}
u(t,x)&=&v(t,x)+\sum_{k=1}^{N_0}\int^t_0 \int^s_0
\phi(\Delta)T_{t-s}
g^k (r, x)dw^k_r ds \\
&=&v(t,x)-\sum_{k=1}^{N_0}\int^t_0\int^t_r\frac{\partial}{\partial
s}T_{t-s}g^k (r, x) ds dw^k_r \\
&=&\sum_{k=1}^{N_0} \int^t_0 T_{t-s}g^k (s,x) dw^k_s.
\end{eqnarray*}
Hence,
$$
\phi(\Delta)u(t,x)=\sum_{k=1}^{N_0}\int^t_0
\phi(\Delta)^{1/2}T_{t-s}\phi(\Delta)^{1/2}g^k (s,\cdot) (x)dw^k_s,
$$
and by
 Burkholder-Davis-Gundy's
inequality, we have
$$
\E \left[  \big|\phi(\Delta)u(t,x)\big|^p \right] \leq c(p)\E\left[
\left(\int^t_0 \sum_{k=1}^{N_0}
|\phi(\Delta)^{1/2}T_{t-s}\phi(\Delta)^{1/2}g^k(s,\cdot) (x)|^2ds\right)^{p/2}\right].
$$
Also, similarly we get
$$
\E \left[  \big|u(t,x)\big|^p \right] \leq c(p)\E\left[
\left(\int^t_0 \sum_{k=1}^{N_0}
|T_{t-s}g^k(s,\cdot) (x)|^2ds\right)^{p/2}\right].
$$
Now it is enough to  use Theorem \ref{main theorem} and Lemma \ref{lemma banach} to conclude
\begin{equation}
               \label{eqn 1.3.6}
\|u\|_{\bH^{\phi,2}_p(T)}\leq N\|g\|_{\bH^{\phi,1}_p(T,\ell_2)}.
\end{equation}

For general $g$, take a sequence $g_n$ so that $g_n \to g$ in $\bH^{\phi,1}_p(T,\ell_2)$ and  each $g_n$ satisfies above described conditions.
Then, by the above result, for $u_n:=\int^t_0 T_{t-s} g^k_n dw^k_s$, we have $du_n=\phi(\Delta)u_n dt +g^k_n dw^k_t$, and
$$
\|u_n\|_{\bH^{\phi,2}_p(T)}\leq N\|g_n\|_{\bH^{\phi,1}_p(T,\ell_2)},
$$
$$
\|u_n-u_m\|_{\bH^{\phi,2}_p(T)}\leq N\|g_n-g_m\|_{\bH^{\phi,1}_p(T,\ell_2)}.
$$
Thus $u_n$ is a Cauchy sequence in $\cH^{\phi,2}_p(T)$ and converges to a certain function $u\in  \cH^{\phi,2}_2(T)$.
One easily gets (\ref{eqn 1.3.6}) by letting $n\to \infty$, and by Theorem  \ref{thm banach} it also follows $\E \sup_{t\leq T}\|u_n-u_m\|^p_{L_p} \to 0$ as $n,m\to \infty$. Finally by taking the limit from
$$
(u_n(t), \varphi)=\int^t_0 (\phi(\Delta)u_n, \varphi)ds+\sum_k \int^t_0 (g^k_n,\varphi)dw^k_s, \quad \forall \,t\leq T\, (a.s.)
$$
and remembering  $\E \sup_{t\leq T}\|u_n-u\|^p_{L_p} \to 0$, we prove that $u$ satisfies
$$
(u(t), \varphi)=\int^t_0 (\phi(\Delta)u, \varphi)ds+\sum_k \int^t_0 (g^k,\varphi)dw^k_s, \quad \forall \,t\leq T\, (a.s.)
$$

{\bf{Step 4}}. General case. The uniqueness follows from Step 1. For the existence and the estimate it is enough to add the solutions in Steps 2 and 3.
The theorem is proved.

\end{proof}

\end{document}